\documentclass[oneside,a4paper,reqno,tbtags,12pt]{amsart}
\usepackage[cp1252]{inputenc}
\usepackage[UKenglish]{babel}
\usepackage[colorlinks,pagebackref]{hyperref}
\hypersetup{pdfstartview=FitH,pdfnewwindow=true}
\hypersetup{pdftitle={Integration by parts in higher order variational calculus},pdfdisplaydoctitle=true}
\usepackage[msc-links]{amsrefs}
\makeatletter
\def\ps@firstpage{\ps@plain
  \def\@oddfoot{\normalfont\scriptsize \hfil\thepage\hfil
     \global\topskip\normaltopskip}
  \def\@oddhead{
  \set@logo{%
 \newbox\logobox\setbox\the\logobox=\hbox{Published in:
 \textit{Matematychni Studii
 (\href{http://mathscinet.imath.kiev.ua/mathscinet/search/journaldoc.html?cn=Mat_Stud}{Matematichn\={\i} Stud\={\i}\"{\i}}),}}
 \newdimen\logowidth\logowidth=\wd\the\logobox
 \newbox\refbox\setbox\the\refbox=\hbox{\MR{\ 1686048 (2000b:58031)};
   Zbl~\href{http://www.zentralblatt-math.org/zmath/en/advanced/?q=an:0958.58016&format=complete}{0958.58016}.}
 \parbox[t]{\logowidth}{\box\the\logobox
   \rightline{vol.~11 (1999), No.~1, p.~85--107.}
   \rightline{\box\the\refbox}
   }}\hss}}
\makeatother
\usepackage{fancyhdr}
\usepackage{amssymb}
\usepackage{xy}
\xyoption{arrow}
\xyoption{matrix}
\xyoption{ps}
\UsePSheader{}

\newtheorem{prop}{Proposition}[section]
\newtheorem{lemma}{Lemma}[section]
\theoremstyle{remark}\newtheorem{rem}{Remark}[section]
\theoremstyle{definition}\newtheorem{define}{Definition}[section]

\chardef\atcode=\catcode`\@\catcode`\@=11\def\section{\@startsection{section}{1}%
  \z@{.7\linespacing\@plus\linespacing}{.5\linespacing}%
  {\normalfont\scshape}}\catcode`\@=\atcode

\newcommand{\msp}[1]{{\mspace{#1mu}}}
\newcommand{\ms}{\kern.2em}
\newcommand{\mms}{\kern.1em}
\newcommand{\msi}{\kern.2em}

\let\botsmash\smash
\let\Cal\mathcal
\let\frak\mathfrak
\let\bold\mathbf
\let\Bbb\mathbb
\let\ssize\scriptscriptstyle
\let\tsize\textstyle
\let\smc\scshape
\let\flushpar\noindent
\let\sssize\scriptscriptstyle
\let\dsize\displaystyle
\let\boldkey\boldsymbol

\newcommand{\eufb}[1]{\boldsymbol{{\mathfrak{#1}}}}
\newcommand{\overg}{\bold{\tilde{\eufb g}}}
\newcommand{\subtld}[1]{\vphantom{_{\scriptscriptstyle\#}}\botsmash{#1{\!}_{_{_{\textstyle\char"7E}}}}}

\DeclareFontFamily{U}{euf}{}
  \DeclareFontShape{U}{euf}{m}{n}{%
    <-6> [.95]   eufm5%
    <6-9>[.95]   eufm7%
    <9-> [.95]   eufm10%
    }{}
  \DeclareFontShape{U}{euf}{b}{n}{%
    <-6> [.95]   eufb5%
    <6-9>[.95]   eufb7%
    <9-> [.95]   eufb10%
    }{}
\DeclareFontFamily{U}{eur}{}
  \DeclareFontShape{U}{eur}{b}{n}{%
    <-6> [.95]   eurb5%
    <6-9>[.95]   eurb7%
    <9-> [.95]   eurb10%
    }{}
\DeclareMathAlphabet\eurb{U}{eur}{b}{n}
\DeclareFontFamily{T1}{scriptpv}{\skewchar\font127 }
\DeclareFontShape{T1}{scriptpv}{m}{n}{%
   <->     [14.4]    scriptpv10%
      }{}
\DeclareSymbolFont{scriptpv}{T1}{scriptpv}{m}{n}
\DeclareSymbolFontAlphabet{\fit}{scriptpv}
\author{Roman~Ya.\msi Matsyuk}
\title[INTEGRATION BY PARTS IN HIGHER ORDER
VARIATIONAL CALCULUS]{INTEGRATION BY PARTS AND VECTOR DIFFERENTIAL FORMS IN HIGHER ORDER
VARIATIONAL CALCULUS ON FIBRED MANIFOLDS}
\address{Institute for Applied Problems in Mechanics and Mathematics\\15 Dudayev Str., L'viv, Ukraine}
\email{romko.b.m@gmail.com; matsyuk@lms.lviv.ua}
\urladdr{http://www.iapmm.lviv.ua}
\begin{document}
\vspace*{17truemm}
\begin{abstract}
Infinitesimal variation of Action functional in classical (non-quantum) field
theory with higher derivatives is presented in terms of well-defined
intrinsic geometric objects independent of the particular field which
varies. ``Integration by parts'' procedure for this variation is then
described in purely formal language and is shown to consist in application
of nonlinear Green formula to the vertical differential of the Lagrangian.
Euler-Lagrange expressions and the Green operator are calculated by simple
pull-backs of certain vector bundle valued differential forms associated with
the given variational problem.
\end{abstract}
\maketitle
\section*{Introduction}
Generally posed variational problem demands covariance
with respect to the pseudo\-group of local transformations which
mix the dependent and independent variables. Appropriate
calculus have been developed by many authors, to mention
{\smc dedecker, tulczyjew, vinogradov, zharinov} as some. In physical
field theory, however, only those transformations preserving
the given fibred structure count. This suggests that geometric
objects adapted to this structure may turn out to be useful.
It is our opinion that the first to mention are semi-basic
differential forms with values in vertical fiber bundles.
A similar approach gained support in~\cite{Garcia_Symposia1974} and further in~\cite{Ferraris1985} and~\cite{Munoz-Masque1985}.

The Fr\'echet derivative of the Action functional at $\upsilon $,
where $\upsilon $ belongs to the set of cross-sections of some fibred
manifold $Y$, is an $r^{\text{th}}$-order differential operator in
the space of variations. Integration-by-parts formula (in other terms,~---
the first variation formula) for this operator involves two other
objects, namely the transpose operator and the Green operator.  Their
definition depends on $\upsilon $ implicitly. Two questions arise
therefore: 1)~By means of what configuration does $\upsilon $ enter
into those operators and how to make this dependence upon $\upsilon $
explicit? 2)~Do there exist such intrinsic geometric objects that the
above mentioned operators for each individual $\upsilon $ might be
calculated by means of a simple geometric procedure? We rewrite the
decomposition formula of {\smc kol\'{a}\v{r}}~\cite{KolarNoveMesto} in terms of
vector bundle valued differential forms to answer these questions.
\newpage
Our goal is threefold. First, to present a rigorous (from the point of view
of Global Analysis) computation of the Fr\'echet derivative of the Action
functional; second, to compute in conceptually the same spirit the transpose
and Green operators; third, to make evident that the notion of a vector bundle
valued differential form is best suited to the peculiarities of variational
calculus in the framework of classical field theory.

The paper is organized in the following way.
In Section\ms\ref{Sec1}, which has a preliminary character, we fix notations
related to the
definitions of some actions of base-substituting morphisms upon vector bundle
valued differential forms. These definitions apply naturally to the calculus
of variations because the prolongations of different orders and various pull-backs
to base manifolds happen in that calculus.
Intending to operate with the pull-backs of differential forms which take values in
some vector bundles, one needs to make intensive use of the
philosophy of induced bundles (reciprocal images). We
explain some basic properties of the reciprocal image
functor and the interplay between it and the notion of vector
bundle valued differential forms. Our development is based on~\cite{Theses}.
The accompanying notations allow us to give the adequate appearance to the integration-by-parts formula later
and also they facilitate the formalization of most proves.

In field theory the notion of the Lagrangian quite naturally falls into
the ramification of the concept of semi-basic differential form with respect
to the independent variables. We turn to the discussion of the Lie derivative
of such a form in Section\ms\ref{Sec2}. This supplies us with the adequate tool for the
description of the infinitesimal variation of the Action functional.

In Section\ms\ref{Sec3} the Action
functional is introduced and the meaning of its variation is established.
Further in this sections we give a purely geometric and strictly intrinsic computation
of the Action variation in terms of the fibre differential and the Lie
derivative operators. In Section\ms\ref{Sec4}  the first variational formula
from the point of view of the concept of vector bundle valued differential
forms is discussed. We also deduce this formula from the Fr\'echet derivative
expression by means of a suitable reformulation of Kol\'a\v r decomposition
formula.
\section{Preliminaries on vector bundle valued semi-basic differential forms\label{Sec1}}
\fancyhead[CE,CO]{\slshape SECTION~\thesection: VECTOR FORMS}
\noindent First we develop some general features
of the behavior of vector bundle valued differential forms under inverse
image functor.
\subsection{Action of base-substituting morphisms upon vector bundles and their
cross-sections}
Consider a vector bundle $\chi :W\rightarrow X$  and a
morphism of manifolds ${{\fit g}}:B\rightarrow X$. The reciprocal image
${{\fit g}}^{-1}\chi $ of the fibre bundle $\chi $ is the set $B\times
_{X}W$  of pairs $(b,{\eurb w})$ with ${{\fit g}}(b)=\chi ({\eurb w})$.  We
denote $\chi ^{-1}({{\fit g}})$  the projection from $B\times _{X}W$  onto
$W$. Given another vector bundle $\chi ^\prime :W^\prime
\rightarrow X^\prime $  and a vector bundle homomorphism ${{\mathsf k}}
:W\rightarrow W^\prime $  over the morphism of bases
${{\fit k}}:X\rightarrow X^\prime $  the homomorphism ${{\mathsf k}}$ may be
reduced to the homomorphism ${{\mathsf k}}_{\sssize X}:W\rightarrow
{{\fit k}}^{\,-1}W^\prime $  over the base $X$ and we denote by
$$
{{\fit g}}^{-1}{{\mathsf k}}\doteqdot {{\fit g}}^{-1}{{\mathsf k}}_{\sssize X}
$$
\noindent the reciprocal image of ${{\mathsf k}}_{\sssize X}$ with respect to ${{\fit g}}$;
it maps ${{\fit g}}^{-1}W$ into ${{{\fit g}}^\prime }^{-1}W'$, where ${{\fit g}}^\prime
={{\fit k}}\circ {{\fit g}}\,$,  and it commutes with ${{\mathsf k}}_{\sssize X}$
via the pair of projections $\chi
^{-1}{{\fit g}}$ and ${\chi ^\prime }^{-1}{{\fit g}}^\prime $ (see Fig.\mms\ref{Figure12} at page\ms\pageref{Figure12}).

Denote by $\pmb{\tilde{\ }}$ the one-to-one correspondence between
the cross-sections
${\eufb s}\in \Gamma \{{{\fit g}}^{-1}W\}$
and vector
fields ${\frak g}:B\rightarrow W$  along ${{\fit g}}$, so that
$\bold{\tilde{\eufb g}}\in \Gamma \{{{\fit g}}^{-1}W\}$, and, reciprocally,
$\vphantom{_x}\botsmash{\underset{\dsize\char"7E}{\frak s}}:B\rightarrow W$
becomes a morphism along ${{\fit g}}$.
Let, in addition, ${{\eufb w}}\in \Gamma \{W\}$ and let
${{\fit g}}^{-1}{{\eufb w}}$ be the reciprocal image of the cross-section ${\eufb w}$
with respect to ${{\fit g}}$. The following relations hold due to the definitions,
$$
\vphantom{_x}\botsmash{\underset{\dsize\char"7E}{\frak s}}
=(\chi ^{-1}{{\fit g}})\circ {{\eufb s}};\quad
(\chi ^{-1}{{\fit g}})\circ \bold{\tilde{\eufb g}}={\frak g};\quad
{{\fit g}}^{-1}{{\eufb w}}=\pmb{\boldkey(}{{\eufb w}}\circ {{\fit g}}\pmb{\boldkey)\,\bold{\tilde{}}}
\;.$$

Homomorphism ${{\mathsf k}}$ acts upon $\Gamma \{{{\fit g}}^{-1}W\}$  by
$$
{{\mathsf k}}_{\sssize\#}:\Gamma \{{{\fit g}}^{-1}W\}\rightarrow
\Gamma \{{{{\fit g}}^\prime }^{-1}W^\prime \},\qquad
 {{\mathsf k}}_{\sssize\#}{{\eufb s}}=({{\fit g}}^{-1}{{\mathsf k}})\circ {{\eufb s}}
$$
(this of course applies to \; ${{\fit g}}={{\fit {id}}}$  \; too).
This action commutes with the suspension operation~$\pmb{\tilde{\ }}$
in the sense that
${{\mathsf k}}_{\sssize\#}\bold{\tilde{\eufb g}}=\pmb{\boldkey(}{{\mathsf k}}\circ {\frak g}\pmb{\boldkey)\,\bold{\tilde{}}}$
 and also  ${{\mathsf k}}\circ\vphantom{_x}\botsmash{\underset{\dsize\char"7E}{\frak s}}
=({{\mathsf k}}_{\sssize\#}{{\eufb s}})_{_{\dsize\char"7E}}$. One could also write
$$
{{\mathsf k}}_{\sssize\#}=({{\fit g}}^{-1}{{\mathsf k}}_{\sssize X})_{\sssize\#}
$$
since ${{\fit g}}^{-1}{{\mathsf k}}_{\sssize X}$  maps over the identity in
$B$.\kern0pt\footnotemark[2]\footnotetext[2]{\kern3pt
This definition and the notation used generalize those of {\smc sternberg}~\cite{Sternberg1964}}
Given one more vector bundle $E\rightarrow B$  and a vector bundle
homomorphism ${{\mathsf g}}:E\rightarrow W$  over ${{\fit g}}$, homomorphism
${{\mathsf g}}^\prime ={{\mathsf k}}\circ {{\mathsf g}}$  acts from $\Gamma
\{E\}$  to $\Gamma \{{{{\fit g}}^\prime }^{-1}W^\prime \}$.

To prove the legitimacy of the usual covariance property,
$$
({{\mathsf k}}\circ {{\mathsf g}})_{\sssize\#}={{\mathsf k}}_{\sssize\#}{{\mathsf g}}_{\sssize\#}
\;,$$
one applies to both sides of it the projection ${\chi
^\prime }^{-1}{{\fit g}}^\prime$ (which is one-to-one on the fibers).
Homomorphism ${{\mathsf k}}\circ {{\mathsf g}}$  acts through the composition
with $({{\mathsf k}}\circ {{\mathsf g}})_{\sssize B}$  and we have ${\chi ^\prime
}^{-1}{{\fit g}}^\prime \circ ({{\mathsf k}}\circ {{\mathsf g}})_{\sssize B}={{\mathsf k}}\circ
{{\mathsf g}}$ (see again Fig.\mms\ref{Figure12}).  On the other hand, by ${{{\fit g}}^\prime
}^{-1}W'\approx {{\fit g}}^{-1}{{\fit k}}^{\,-1}W^\prime $ we have ${\chi ^\prime
}^{-1}{{\fit g}}^\prime \circ {{\fit g}}^{-1}{{\mathsf k}}_{\sssize X}\circ {{\mathsf g}}
_{\sssize B} \approx {\chi ^\prime }^{-1}{{\fit k}}\circ ({{\fit k}}^{\,-1}\chi ^\prime
)^{-1}{{\fit g}}\circ {{\fit g}}^{-1}{{\mathsf k}}_{\sssize X}\circ {{\mathsf g}}
_{\sssize B} = {\chi ^\prime }^{-1}{{\fit k}}\circ {{\mathsf k}}_{\sssize X}\circ
\chi ^{-1}{{\fit g}}\circ {{\mathsf g}}_{\sssize B} = {{\mathsf k}}\circ {{\mathsf g}}$,  q.e.d.
\subsection{Action of base-substituting morphisms upon vector bundle valued differential
forms}
A differential form on $B$ with values in a vector bundle $E$ is a
cross-section of the vector bundle $E \otimes \wedge T^\ast B$. Of
course, one can take ${{\fit g}}^{-1}W$ in place of $E$ and speak of differential
forms which take values in $W$. We shall use the notation $\Omega
^{d}(B;E)$  for $\Gamma \{E\otimes \wedge\!^{d}T^\ast
B\}$  and also write sometimes $\Omega (B;W)$  instead of $\Omega
(B;{{\fit g}}^{-1}W)$. If we take $\wedge E^\ast $  in place of $E$, the
module $\Omega (B;\wedge E^\ast )$ acquires the structure of a
bigraded algebra.  If (locally) $\boldsymbol\omega _{l}=\boldsymbol\varphi
_{l}\otimes \boldsymbol\alpha _{l}$, $l=1,2$, then $\boldsymbol\omega _{1}\wedge
\boldsymbol\omega _{2}=\boldsymbol\varphi _{1}\wedge \boldsymbol\varphi _{2}\otimes
\boldsymbol\alpha _{1}\wedge  \boldsymbol\alpha _{2}$. The interior product
${\bold i}:\Gamma \{E\}\times \Gamma \{\wedge E^\ast \}\rightarrow
\Gamma \{\wedge E^\ast \}$  defines a coupling $\wedge _{{\bold i}}$
from $\Omega (B;E)\times \Omega (B;\wedge\!^{d}E^\ast )$  into
$\Omega (B;\wedge ^{d-1}E^\ast )$;  if $\boldsymbol\rho ={{\eufb e}}
\otimes  \boldsymbol\alpha $, ${{\eufb e}}\in \Gamma \{E\}$ and
$\boldsymbol\alpha\in \Gamma \{\wedge T^\ast B\}$, then $\boldsymbol\rho
\wedge _{{\bold i}}\boldsymbol\omega _{1}={\bold i}({{\eufb e}}
)\boldsymbol\varphi_{1} \otimes \boldsymbol\alpha\wedge \boldsymbol\alpha _{1}$.  We
denote $\left<\boldsymbol\rho ,\boldsymbol\omega \right>\doteqdot \boldsymbol\rho
\wedge _{{\bold i}}\boldsymbol\omega $ when $\boldsymbol\omega \in \Omega
(B;E^\ast )$.

Dual ${{\fit k}}$-comorphism $({{\mathsf k}}_{\sssize X})^\ast $ acts upon $\Omega (B;
W^{\prime\ast})$  by composition with vector bundle homomorphism
${{\fit g}}^{-1}{{{\mathsf k}}_{\sssize X}}^\ast \otimes {{\mathsf{id}}}$.  We write
${{\mathsf k}}^{\sssize\#}\boldsymbol\theta^{\boldsymbol\prime}$  for the cross-section
$({{\fit g}}^{-1}{{{\mathsf k}}_{\sssize X}}^\ast \otimes {{\mathsf{id}}})\circ
\boldsymbol\theta^{\boldsymbol\prime}$  and it is of course true that
\begin{equation}
\left<{{\mathsf k}}_{\sssize\#}{{\eufb s}} ,\boldsymbol\theta^{\boldsymbol\prime}
 \right>=\left<{{\eufb s}} ,{{\mathsf k}}^{\sssize\#}\boldsymbol\theta^{\boldsymbol\prime}
 \right>
\label{1}
\end{equation}
for ${{\eufb s}} \in \Omega^{0}(B;W)$  and $\boldsymbol\theta^{\boldsymbol\prime} \in \Omega (B;W^{\prime\ast})$.

Inspired by the considerations, heretofore delivered, we introduce some definitions.
\begin{define}\label{Definition1}
Let ${{\fit k}}$  be a morphism of manifolds from $X$
into $X^\prime $, and let $W$  and $W'$  be fibred
manifolds over $X$ and $X^\prime $ respectively.

If there exists a natural lift of ${{\fit k}}$ to a fibred
morphism
${\mathsf F}_{{\fit k}}:W\rightarrow W^\prime $, then ${{\fit k}}$
acts upon $\Gamma \{W\}$ as follows
$$
{{\fit k}}_{\sssize\#}:\Gamma \{W\}\rightarrow \Gamma \{{{\fit k}}^{\,-1}W^\prime \},
\qquad {{\fit k}}_{\sssize\#}{\eufb w}=\pmb{\boldkey(}{\mathsf F}_{{\fit k}}\circ {\eufb w}\pmb{\boldkey)\,\bold{\tilde{}}}
\;.$$

Let $W$ and $W^\prime $ be vector bundles and assume a couple of morphisms
${{\fit g}}:B\rightarrow X$ and ${{\fit g}}^\prime :B\rightarrow X^\prime $
be given such that
${{\fit k}} \circ {{\fit g}}={{\fit g}}^\prime $. Then
the following modules an operators between them are defined:

\begin{gather*}
\boxed{
{{\fit k}}_{\sssize\#}
:\Omega (B;W)\rightarrow \Omega (B;W^\prime ),
\qquad {{\fit k}}_{\sssize\#}\boldsymbol\eta
=({{\fit g}}^{-1}{\mathsf F}_{{\fit k}}\otimes{{\mathsf{id}}})_{\sssize\#}\boldsymbol\eta
}
\\
\boxed{
{{\fit k}}^{\sssize\#}
:\Omega(B;W^{\prime\ast})\rightarrow \Omega (B;W^\ast ),
\qquad{{\fit k}}^{\sssize\#}\boldsymbol\theta^{\boldsymbol\prime}
=({{\fit g}}^{-1}{({\mathsf F}_{{\fit k}})_{_X}}^\ast\otimes{{\mathsf{id}}})_{\sssize\#}\boldsymbol\theta^{\boldsymbol\prime}
}
\end{gather*}

If there does not exist but merely a ${{\fit k}}$-comorphism
${\mathsf F}^\ast _{{\fit k}}$ from ${{\fit k}}^{\,-1}W^\prime $
into $W$
over the identity in $X$ then the action ${{\fit k}}^{\sssize\#}$
still can be defined as
\begin{equation}
\boxed{
{{\fit k}}^{\sssize\#}:\Omega (B;W^\prime )\rightarrow \Omega (B;W),
\qquad {{\fit k}}^{\sssize\#}\boldsymbol\eta^{\boldsymbol\prime}
=({{\fit g}}^{-1}{\mathsf F}^\ast _{{\fit k}} \otimes {{\mathsf{id}}})_{\sssize\#}\boldsymbol\eta^{\boldsymbol\prime}
}\label{2}
\end{equation}
Of course, one may set
${{\fit g}}={{\fit {id}}}$  everywhere in the above.
\end{define}

Assume be given a morphism $\delta $ from a manifold
$Z$  into B. We recall that the definition of the reciprocal
image $\delta ^{-1}\boldsymbol\rho $  applies to a cross-section
$\boldsymbol\rho \in \Gamma \{E\otimes \wedge T^\ast B\}$
by means of
$$
\delta ^{-1}:\Omega (B;E)\rightarrow \Gamma \{\delta
^{-1}E\otimes \delta ^{-1}\wedge T^\ast B\}, \qquad
\delta^{-1}\boldsymbol\rho =\pmb{\boldkey(}\boldsymbol\rho \circ \delta \pmb{\boldkey)\,\bold{\tilde{}}}
\;.$$

Let us accept the following brief notations for the mappings, induced over $Z$
by the tangent functor $T$,
$$
\delta ^{\sssize T}\doteqdot (T\delta )_{\sssize Z}\,,\qquad
\delta ^\ast \doteqdot (\delta ^{\sssize T})^\ast
\;.$$

\begin{define}\label{Definition2}
Let a vector bundle $F^\prime $ over $Z$ be given.
The action $\delta ^{(\cdot,{\sssize\#})}$ of the morphism $\delta $ upon the module
$\Gamma \{F^\prime \otimes \delta ^{-1}\wedge T^\ast B\}$
is defined by
$$
\delta ^{(\cdot,{\sssize\#})}:\Gamma \{F^\prime \otimes \delta ^{-1}\wedge T^\ast
B\}\rightarrow \Omega (Z;F^\prime ), \qquad
\delta ^{(\cdot,{\sssize\#})}=({{\mathsf{id}}}\otimes \wedge \delta ^\ast )_{\sssize\#}
\;.$$
\end{define}

The pull-back of a differential form $\boldsymbol\rho \in \Omega (B;E)$
is hereupon defined by (see also~\cite{Bourbaki1971})
$$
\delta ^{\star}:\Omega (B;E)\rightarrow \Omega (Z;E), \qquad
\delta ^{\star}\boldsymbol\rho =\delta ^{(\cdot,{\sssize\#})}\delta ^{-1}\boldsymbol\rho
\;.$$
\begin{define}\label{Definition3}
If there exists a natural $\delta $-comorphism
${\mathsf F}^\ast_\delta$
from $\delta ^{-1}E$ to a vector bundle
$F$ over $Z$ then $\delta ^{\sssize\#}$ will mean
the total of the ``twofold'' backward action
$$\boxed{
\delta ^{\sssize\#}:\Gamma \{\delta ^{-1}E\otimes \delta ^{-1}\wedge
T^\ast B\}\rightarrow \Omega (Z;F), \qquad
\delta ^{\sssize\#}=({\mathsf F}^\ast_\delta\otimes \wedge \delta ^\ast
)_{\sssize\#}
}$$
\end{define}

The definitions introduced heretofore correlate. For instance,
a posteriori given some $\delta$-comorphism ${\mathsf F}^\ast_\delta$,
we can apply the operation $\delta ^{\sssize\#}$ in the spirit of the Definition\ms\ref{Definition1}
to the differential form $\delta^{\star}\boldsymbol\rho $ by renaming in (\ref{2})
$B$ as $Z$, $X^\prime$ as $B$, ${{\fit k}}$ as $\delta $, and putting
${{\fit g}}={{\fit {id}}}$, $W=F$, $W^\prime =E$ and
${\mathsf F}^\ast _{{\fit k}}={\mathsf F}^\ast _\delta$,
which will then produce
$$
\delta ^{\sssize\#}\delta ^{\star}\boldsymbol\rho =({\mathsf F}^\ast _\delta\otimes
{{\mathsf{id}}})_{\sssize\#}\delta ^{\star}\boldsymbol\rho =({\mathsf F}^\ast _\delta\otimes
{{\mathsf{id}}})_{\sssize\#} ({{\mathsf{id}}}\otimes\wedge\delta ^\ast
)_{\sssize\#}\delta ^{-1}\boldsymbol\rho =\delta ^{\sssize\#}\delta
^{-1}\boldsymbol\rho
\;,$$
wherein the operation $\delta ^{\sssize\#}$ to the extreme right is defined in the spirit
of the Definition\ms\ref{Definition3} this time.

In what comes later, we shall {\it not\/} bother to indicate explicitly the
``partial'' character of the $\delta ^{(\cdot,{\sssize\#})}$ operation
at the occasions like that of the Definition\ms\ref{Definition2} any more, and shall exploit the same brief
notation $\delta ^{\sssize\#}$ in place of the more informative one, $\delta ^{(\cdot,{\sssize\#})}$,
because hardly any confusion will ever arise.

\begin{define}\label{Definition4}
Given a differential form
$\boldsymbol\theta \in \Omega (B;W^\ast )$ and a cross-section ${\eufb w} \in
\Gamma \{W\}$, we define the contraction $\left<{\eufb w} ,\boldsymbol\theta
\right>$ by $$ \left<{\eufb w} ,\boldsymbol\theta \right>=\left<{{\fit g}}^{-1}{\eufb w}
,\boldsymbol\theta \right>.  $$
\end{define}

One can easily verify that the following formula holds for
$\boldsymbol\omega\in\Omega(B;E^\ast)$ and ${\eufb e}\in\Omega^{0}(B;E)$,
\begin{equation}
\left<\delta^{-1}{\eufb e},\delta^{-1}\boldsymbol\omega\right>
=\delta^{-1}\left<{\eufb e},\boldsymbol\omega\right>
\,.\label{3}
\end{equation}
\subsection{Semi-basic differential forms}
Let $\pi :B\rightarrow Z$  be a surmersion of manifolds.
The reciprocal image ${\pi ^{-1}}{TZ}$ of the tangent bundle $TZ$ is
incorporated in the commutative diagram of vector bundle
homomorphisms of Fig.\mms\ref{Figure1}\,:
\begin{figure}[h]
$$
\xy
\xymatrix{
TB
\ar[dd]^{\tau_{_B}}
\ar[dr]|{\pi^{T}}
\ar[drrr]^-{T\pi}
\\
&\pi^{\scriptscriptstyle-1}(TZ)
\ar[rr]_-{\tau^{-1}\pi}
\ar[dl]|{\object+{\scriptstyle\pi^{-1}\tau}}
&&TZ
\ar[d]^{\tau}
\\
B
\ar[rrr]^\pi
&&&Z
}\endxy$$
\caption{}
\label{Figure1}
\end{figure}

The existence of the short exact sequence of vector bundle
homomorphisms
\begin{equation}
0\rightarrow T(B/Z)\buildrel{\iota}\over\longrightarrow TB
\buildrel{\pi^{\sssize T}}\over\longrightarrow\pi ^{-1}(TZ)\rightarrow 0
\label{4}
\end{equation}
allows us to define the module ${\frak V}_{\sssize B}=\Gamma
\{T(B/Z)\}$  of vertical (with respect to $\pi $)
vector fields on the manifold $B$.
Let ${\frak F}_{\sssize B}$ denote  the ring of $C^{\infty }$ functions over
the manifold $B$. The ${\frak F}_{\sssize B}$-module of all-direction vector fields over the manifold
$B$ will be denoted by ${\frak X}_{\sssize B}$.

Utilizing the partition of unity over the manifold $B$ one
can split the sequence (\ref{4}),
$$
0 \leftarrow T(B/Z) \buildrel{\overleftarrow{\;\iota}}\over\longleftarrow
TB \buildrel{\;\:\overleftarrow{\pi^{\sssize T}}}\over\longleftarrow  \pi^{-1}TZ \leftarrow 0.
$$

In fact, the restriction of the vector bundle homomorphism
$\overleftarrow{\iota}$ to the subbundle $\text{Im}\,\iota $ is the inverse to
the mapping $\iota $.  If $\eurb t=\iota (\eurb v)\in
\text{Im}\,\iota $, then $\iota \circ \overleftarrow{\iota} (\eurb t)
=\iota \circ \overleftarrow{\iota} \circ \iota (\eurb v) =\iota
({\eurb v})={\eurb t}$ and so $\iota \circ \overleftarrow{\iota}
={{\mathsf{id}}}$, q.e.d.

The reciprocal image $\pi ^{-1}(T^\ast Z)\approx (\pi
 ^{-1}(TZ))^\ast $ of the cotangent bundle $T^\ast Z$ is
incorporated in the commutative diagram of Fig.\mms\ref{Figure2}:
\begin{figure}[h]
$$
\xy
\xymatrix{
\wedge T^\ast B
\ar[dr]|{\object+{\scriptstyle{\overset\ast\tau}_{_B}}}
&&\pi^{\scriptscriptstyle-1}(\wedge T^\ast Z)
\ar[ll]_-{\wedge\pi^\ast}
\ar[rr]^-{\overset\ast\tau{}^{-1}\pi}
\ar[dl]|{\object+{\scriptstyle\pi^{-1}\overset\ast\tau}}
&&\wedge T^\ast Z
\ar[d]^{\overset\ast\tau} \\
&B\ar[rrr]^\pi
&&&Z
}\endxy
$$
\caption{}
\label{Figure2}
\end{figure}

Comorphism $\wedge \pi ^\ast $  is dual to the morphism
$T\pi $  in the sense that $\wedge \pi ^\ast =\wedge (\pi
^{\sssize T})^\ast $. We denote the effect of $\wedge \pi ^\ast $  on the
sections of the induced bundle $\pi ^{-1}\wedge T^\ast Z$  by $\pi
^{\sssize\#}$, so $\pi ^{\sssize\#}\boldsymbol\beta =(\wedge\!^{d}\pi ^\ast )\circ
\boldsymbol\beta$ if $\boldsymbol\beta \in \Gamma \{\pi ^{-1}\wedge\!^{d}T^\ast
Z\}$.

Let us show that the sequence of the homomorphisms of the modules
of cross-sections, which corresponds to the exact sequence (\ref{4}),
\begin{equation}
0\rightarrow {\frak V}_{\sssize B}\buildrel{\iota _{\sssize\#}}
\over\longrightarrow{\frak X}_{\sssize B}\buildrel{\pi _{\sssize\#}}
\over\longrightarrow\Gamma \{\pi ^{-1}TZ\}\rightarrow 0,
\label{5}
\end{equation}
is exact as well.

Let ${\eufb x}\in  \text{Ker}\,\pi_{\sssize\#}$. The exactness of the
sequence (\ref{4})
implies that ${\eufb x}\in \text{Im}\,\iota $. Then the cross-section
$\overleftarrow{\iota}\circ {\eufb x}$  is being mapped into ${\eufb x}$ under
the homomorphism $\iota _{\sssize\#}$, because we have $\iota
_{\sssize\#}(\overleftarrow{\iota}\circ {\eufb x})\equiv \iota \circ
\overleftarrow{\iota} \circ {\eufb x}={\eufb x}$.  Thus ${\eufb x}\in
\text{Im}\,\iota_{\sssize\#}$  and so $\text{Im}\,\iota _{\sssize\#}\supset \text{Ker}\,\pi_{\sssize\#}$.
Examining the surjectivity of $\pi _{\sssize\#}$ one sees easily that for
every ${\eufb h}\in \Gamma \{\pi^{-1}TZ\}$ the cross-section
$\overleftarrow{\pi^{\sssize T}}\circ {\eufb h}$
is mapped by the homomorphism $\pi
_{\sssize\#}$ into the cross-section ${{\eufb h}}$ because of $\pi
_{\sssize\#}(\overleftarrow{\pi^{\sssize T}}\circ {\eufb h})\equiv \pi \circ
\overleftarrow{\pi^{\sssize T}}\circ {\eufb h}={\eufb h}$. The injectivity
of the homomorphism $\iota _{\sssize\#}$ and the inclusion ${\text{Im}\,\iota
_{\sssize\#}}\subset {\text{Ker}\,\pi _{\sssize\#}}$  are still more obvious,
q.e.d.

Due to the exactness of the sequence (\ref{5}) one can identify the module
$\Gamma \{\pi ^{-1}TZ\}$  with the quotient module ${\frak X}_{\sssize B}/{\frak V}_{\sssize B}$.
We introduce the notation $\Omega (B)$
for the graded algebra of differential forms, $\Omega (B)=\sum
^{\dim B}_{d=0}\Omega ^{d}(B)$, so that ${\frak F}_{\sssize B}=\Omega
^{0}(B)$.  Let also ${\bold A}\msp1\negmedspace^{d}$ mean the functor of skew-symmetric
multilinear forms of degree $d$ on some module. For a vector bundle
$E$ we shall exploit the moduli isomorphism ${\bold A}\msp1\negmedspace^{d}{}\bigl(\Gamma
\{E\}\bigr)\approx \Gamma \{\wedge\!^{d}{}E^\ast \}$.
As we have just seen, the
${\frak F}_{\sssize B}$-algebras
${\bold A}({\frak X}_{\sssize B}/{\frak V}_{\sssize B})$  and
$\Gamma \{\pi ^{-1}\wedge T^\ast Z\}$ can be identified with each other.
Cross-sections of the bundle $\pi ^{-1}\wedge T^\ast Z$  are known
as semi-basic (with respect to $\pi )$ differential forms on the
manifold $B$. The graded algebra of these forms will be denoted
as $\Omega_{\sssize B}(Z)$. Remind that the exactness of the sequence (\ref{4})
means that the vector bundles $\pi ^{-1}TZ$  and
$TB/T(B/Z)$ are isomorphic. Passing to the dual bundles we obtain the
moduli isomorphism $\Gamma \{\pi ^{-1}\wedge T^\ast Z\}\approx \Gamma
\bigl\{\wedge \bigl(TB/T(B/Z)\bigr)^\ast \bigr\}$  and finish up with the double
isomorphism of ${\frak F}_{\sssize B}$-algebras
$$
\boxed{\Omega ^{d}_{\sssize B}(Z)\approx {\bold A}\msp2\negmedspace^{d}({\frak X}_{\sssize B}/
{\frak V}_{\sssize B})\approx \Gamma \bigl\{\wedge\!^{d}\bigl(TB/T(B/Z)\bigr)^\ast \bigr\}
}$$

The elements of the middle-term algebra will hereinafter be called
the horizontal (with respect to $\pi )$ differential forms and they
will be identified by means of the second isomorphism with such forms
$ \boldsymbol\alpha \in \Omega ^{d}(B)$
  that
$ \boldsymbol\alpha ({\eurb t}_{1},\ldots,{\eurb t}_{d})=0$
every time when at least one of the tangent
vectors ${\eurb t}_{1},\ldots ,{\eurb t}_{d}$  is vertical.

Let $ \boldsymbol\alpha \in {\bold A}\msp1\negmedspace^{d}{}({\frak X}_{\sssize B}/{\frak V}_{\sssize B})$. The
differential form $\boldsymbol\beta \in \Omega ^{d}_{\sssize B}(Z)$, such that $\boldsymbol\beta
({{\eufb h}}_{1},\ldots ,{{\eufb h}}_{d})= \boldsymbol\alpha ({{\eufb x}}_{1},\ldots
,{{\eufb x}}_{d})$ if ${{\eufb h}}_{i}=\pi _{\sssize\#}{{\eufb x}}_{i}$, is the
image of $ \boldsymbol\alpha $ under the isomorphism
${{\bold A}\msp1\negmedspace^{d}}{}({\frak X}_{\sssize B}/{\frak V}_{\sssize B})\approx
\Omega^{d}_{\sssize B}(Z)$.  Hereinafter we shall use one and the same notation
$\Omega_{\sssize B}(Z)$ for both algebras and we shall write $\Omega
_{r}(Z)$  when the manifold $Y_{r}$ will be considered in place of
$B$. We also introduce a separate notation ${\frak H}_{\sssize B}\equiv
{\frak H}_{\sssize B}(Z)$ for the module of cross-sections of the bundle $\pi
^{-1}TZ$; thus $\Omega ^{1}_{\sssize B}(Z)={{\frak H}_{\sssize B}}^\ast$. The
elements ${{\eufb h}}$ of the module ${\frak H}_{\sssize B}$ by means of the
vector bundle homomorphism $\tau^{-1}\pi $ are identified with the
corresponding lifts
$\vphantom{_x}\botsmash{\underset{\dsize\char"7E}{\botsmash{\frak h}}}$
of the morphism $\pi $, conventionally known as vector fields along $\pi $.
\section{Infinitesimal variations and the Lie derivative\label{Sec2}}
\fancyhead[CE,CO]{\slshape SECTION~\thesection: LIE DERIVATIVE}
\subsection{The Fr\'echet derivative of the base substitution}
Let $\delta :Z\rightarrow B$ be a morphism of manifolds and
let $\boldsymbol\alpha \in \Omega (B)$ be a differential form on the manifold $B$.
Assume $Z$ compact. The map $^{\star}\boldsymbol\alpha :C^{\infty }(Z,B)\rightarrow \Omega (Z)$
takes any morphism $\delta $ over to the differential form
$\delta ^{\star}\boldsymbol\alpha \in \Omega (Z)$.
The space $C^{\infty }(Z,B)$ of
$C^{\infty }$ mappings from $Z$ into $B$ has as its tangent vector at
the point
$\delta \in C^{\infty }(Z,B)$
some lift
$\frak b:Z\rightarrow TB$
of the morphism $\delta $.
Consider a one-parametric family $\delta _{t}$
of deformations of the morphism
$\delta$, $\delta _{t}\in C^{\infty }(Z,B)$,
$\delta _{0}=\delta $. The mapping
$\gamma _{\delta}:t\rightarrow\delta _{t}$
defines a smooth curve in the manifold
$C^{\infty }(Z,B)$. Let the lift $\frak b$ be
the tangent vector to this curve at the point $\delta $. Accordingly
to the definition of the tangent map, the Fr\'echet derivative of the
application $^{\star}\boldsymbol\alpha$ at the point $\delta $ is being evaluated
on the tangent vector $\frak b$ in the following way:
\begin{equation}\label{matsyuk:D*alpha}
(\,{\bold D}\,{^{\star}\boldsymbol\alpha}\,)(\delta)\,{\boldkey.}\,
\frak b
=(d/dt)(\,{^{\star}\boldsymbol\alpha}\,\circ \,\gamma _{\delta}\,)(0).
\end{equation}
\subsection{The Fr\'echet derivative  and the Lie derivative}
The operation of the Lie derivative in the direction of the lift
$\frak b=\tau^{-1}\delta \circ \bold{\tilde{\eufb b}}$
will be introduced via the formula (compare with~\cite{Bourbaki1971})
$$
\boxed{{\bold L}(\frak b)
={\bold d}{\delta}^{\sssize\#}{\bold i}(\bold{\tilde{\eufb b}})\delta^{-1}+
\delta ^{\sssize\#}{\bold i}(\bold{\tilde{\eufb b}})\delta^{-1}{\bold d}
}$$
which obviously generalizes the conventional one. The
derivation
${\bold i}(\bold{\tilde{\eufb b}})$ is defined in terms of the interior
product of the cross-section
$\bold{\tilde{\eufb b}}\in\Gamma \{\delta ^{-1}TB\}$.
The differential form
$\delta ^{-1}\boldsymbol\alpha$ is in $\Gamma \{\delta^{-1}\wedge T^\ast B\}$
whenever $\boldsymbol\alpha \in \Omega (B)$.

The proof of the formula
$$
\boxed{{\bold L}(\frak b)\boldsymbol\alpha
=(d/dt)\ {^{\star}\boldsymbol\alpha }\circ \gamma _\delta\ (0)
}$$
closely follows the lines of the proof of the corresponding analogue
for the conventional Lie derivative as follows:
\begin{proof}
It suffices to verify
the effect of the Lie derivative action upon functions and Pfaff forms alone since
${\bold L}(\frak b)$ acts as a derivation.

Indeed, first we convince ourselves that the construction
$\delta ^{\sssize\#}{\bold i}(\bold{\tilde{\eufb b}})\delta ^{-1}$
acts as a derivation of degree -1,
\begin{align*}
\delta ^{\sssize\#}{\bold i}(\bold{\tilde{\eufb b}})\delta ^{-1}(\boldsymbol\alpha \wedge
\boldsymbol\alpha^{\boldsymbol\prime} )&=\delta ^{\sssize\#}{\bold i}(\bold{\tilde{\eufb b}})
(\delta ^{-1}\boldsymbol\alpha \wedge \delta ^{-1}\boldsymbol\alpha^{\boldsymbol\prime} ) \\
&= \delta^{\sssize\#}{\bold i}(\bold{\tilde{\eufb b}})\delta^{-1}\boldsymbol\alpha\wedge
\delta^{\star}{\boldsymbol\alpha^{\boldsymbol\prime}}
+(-1)^{\deg (\boldsymbol\alpha )}\delta ^{\star}\boldsymbol\alpha
\wedge \delta ^{\sssize\#}{\bold i}(\bold{\tilde{\eufb b}})\delta ^{-1}\boldsymbol\alpha^{\boldsymbol\prime} .
\end{align*}
Then we remind that the exterior differential ${\bold d}$
acts as a derivation of degree  +1 and finally we notice that the Lie
derivative ${\bold L}(\frak b)$ appears to be their
commutator and by this fact is forced to act as a derivation of
degree  0  from the algebra $\Omega (B)$ into the algebra
$\Omega (Z)$ along the homomorphism $\delta^{\star}$,
$$
{\bold L}(\frak b)(\boldsymbol\alpha \wedge \boldsymbol\alpha^{\boldsymbol\prime} )
={\bold L}(\frak b)\boldsymbol\alpha\wedge\delta
^{\star}{\boldsymbol\alpha^{\boldsymbol\prime} }+\delta^{\star}\boldsymbol\alpha
\wedge {\bold L}(\frak b) \boldsymbol\alpha^{\boldsymbol\prime}
$$

The operator $(d/dt){{\delta_t}^{\star}}$ also acts as a derivation,
$$
(d/dt){{\delta_t}^{\star}}(\boldsymbol\alpha \wedge \boldsymbol\alpha^{\boldsymbol\prime} )
=(d/dt){{\delta_t}^{\star}}\boldsymbol\alpha \wedge
{\delta_t}^{\star}\boldsymbol\alpha^{\boldsymbol\prime} +{\delta_t}^{\star}
\boldsymbol\alpha \wedge (d/dt){\delta_t}^{\star}\boldsymbol\alpha^{\boldsymbol\prime} .
$$
So one
concludes that these operators coincide.
\end{proof}

The The Fr\'echet derivative~(\ref{matsyuk:D*alpha}) thus may be expressed by the
following computational formula,
\begin{equation}
\boxed{({\bold D}{^{\star}\boldsymbol\alpha })(\delta )\,{\boldkey.}\,
\frak b={\bold L}(\frak b)\,\boldsymbol\alpha
}
\label{6}
\end{equation}
\subsection{Fibre differential}
A semi-basic differential form $\boldsymbol\beta \in \Omega_{\sssize B}(Z)$  may be
considered as a fibred manifolds morphism
$\vphantom{_x}\botsmash{\underset{\dsize\char"7E}{\botsmash{\beta}}}:B\rightarrow
\wedge T^{{\ast }}Z$
over $Z$. In a more general way, let
$\zeta :F\rightarrow Z$  be a vector bundle and let
$\vphantom{_x}\botsmash{\underset{\dsize\char"7E}{\botsmash{\beta}}}:B\rightarrow F$
be a fibred morphism over the base
$Z$, \,$\zeta \circ \vphantom{_x}\botsmash{\underset
{\dsize\char"7E}{\botsmash{\beta}}}=\pi $ (see Fig.\mms\ref{Figure3} below).
{\begin{figure}[h]
\def\ubeta{\vphantom{_x}\botsmash{\underset{\dsize\char"7E}{\botsmash{\beta}}}}
$$\xy
\xymatrix{
&
\save \POS+<-2pt,-14pt>\drop+{TB}\POS="co"\restore
&&
\save \POS+<2pt,-14pt>\drop+{TF}\POS="com"\restore
\ar"co";"com"|{\object+{\scriptstyle T\ubeta}}
\\
VB\ar@{_(->}+UR+<-2\jot,2\jot>;"co"^-{\iota}
\ar[rrrr]|-{\object+{\scriptstyle V\ubeta}}
\ar[dddr]_{\tau_{_{B}}}
&&&&VF\approx\zeta^{\scriptscriptstyle -1}F
\ar@{^(->}+UL+<2\jot,2\jot>;"com"_-{\iota}
\ar[dd]|<>(.367){\object+{}}
\\
&&
\save \POS+<-15pt,7pt>
   \drop+{\pi^{\scriptscriptstyle -1}F\approx\ubeta^{\scriptscriptstyle -1}(VF)}
   \ar[ddl]
   \ar+UR;[urr]+DL
   \POS="comm"
   \ar[ull];"comm"+UL+<+3pt,-3pt>^-{\ubeta^V}
   \restore
&& &&
\save \POS+<0pt,7pt>
   \drop+{F}
   \ar[dd]_{\zeta}
   \POS="comment"
   \restore
   \ar"comm";"comment"^(.65){\zeta^{-1}(\pi)}
   \ar[ull];"comment"^-{\zeta^{-1}(\zeta)}
\\
&&&&F
\ar[drr]^{\zeta}
\\
&B\ar[urrr]^{\ubeta}
\ar[rrrrr]^{\pi}
&&&&&Z
}\endxy$$
\caption{}
\label{Figure3}
\end{figure}
}The restriction of the tangent mapping
$T\vphantom{_x}\botsmash{\underset{\dsize\char"7E}{\botsmash{\beta}}}$  to the bundle of
vertical tangent vectors gives rise to the vector bundle homomorphism
$V\vphantom{_x}\botsmash {\underset{\dsize\char"7E}
{\botsmash{\beta}}}:VB\rightarrow VF$  over the
morphism $\vphantom{_x}\botsmash{\underset
{\dsize\char"7E}{\botsmash{\beta}}}$  (we use more eco\-no\-mi\-cal
notations $VB$ and $VF$ for the bundles of vertical tangent vectors $T(B/Z)$
and $T(F/Z)$ over $Z$). Let ${\sigma _{\eurb f}}^\prime $ denote the
tangent vector to the curve $\sigma _{\eurb f}(t)$  which belongs
completely to the fibre $F_{z}$ of $F$ over $z\in Z$  and which
starts from
${\eurb f}\in F$, $\sigma _{\eurb f}(0)={\eurb f}$;
then the derivative $(d\sigma _{\eurb f}/dt)(0)$  also belongs to
the vector space $F_{z}$. The vertical tangent vector
${\sigma_{\eurb f}}^\prime $
is identified with the pair
$\bigl({\eurb f}; (d\sigma _{\eurb f}/dt)(0)\bigr)$
of the induced bundle $\zeta ^{-1}F$
and so the well-known isomorphism $VF\approx \zeta ^{-1}F$  over
$F$  holds. The homomorphism $V\vphantom{_x}\botsmash{\underset{\dsize\char"7E}{\botsmash{\beta}}}$
may be reduced to the base $B$ and the morphism so defined,
$(\vphantom{_x}\botsmash{\underset{\dsize\char"7E}{\botsmash{\beta}}})^{\sssize V}:
VB\rightarrow {\vphantom{_x}\botsmash{\underset{\dsize\char"7E}{\botsmash{\beta}}}}^{-1}VF$,
after the identification of
${\vphantom{_x}\botsmash{\underset{\dsize\char"7E}{\botsmash{\beta}}}}^{-1}VF\approx
{\vphantom{_x}\botsmash{\underset{\dsize\char"7E}{\botsmash{\beta}}}}^{-1}\zeta ^{-1}F$
with $\pi ^{-1}F$  acts
upon a vertical tangent vector ${\eurb v\in}VB$  as follows.
Suppose vector ${\eurb v}$ be tangent to a curve $\sigma _{b}$ in the
manifold $B$ and $b=\sigma _{b}(0)$. Then
${\vphantom{_x}\botsmash{\underset{\dsize\char"7E}{\botsmash{\beta}}}}^{\sssize V}({\eurb v})=\bigl(b;
(d/dt)({\vphantom{_x}\botsmash{\underset{\dsize\char"7E}{\botsmash{\beta}}}} \circ \sigma _{b})(0)\bigr)$.
This mapping ${\vphantom{_x}\botsmash{\underset{\dsize\char"7E}{\botsmash{\beta}}}}^{\sssize V}$ is linear at the
fibers of $F$ and may be thus thought of as a cross-section
${\bold d}_{\pi }\boldsymbol\beta $  of the bundle
$(VB)^{{\ast }}\otimes \pi ^{-1}F$,
\begin{equation}
\boxed{\left<{\eufb v},{\bold d}_{\pi }\boldsymbol\beta
\right>\,\approx\,
\pmb{\boldkey(}\,V\vphantom{_{_{_{_x}}}}\botsmash{\underset{\dsize\char"7E}{\beta}}\,\circ\,
{\eufb v}\,\pmb{\boldkey)\,\bold{\tilde{}}}
}\label{7}
\end{equation}

\subsubsection*{\indent The Lie derivative and fibrewise differentiation of semi-basic forms}
In what follows and to the end of current Section we shall be busy with
establishing the relationship between the Lie derivative and the fibre
differential of a semi-basic differential form

Consider a vector bundle $E\rightarrow B$ and its dual bundle $E^\ast\rightarrow B$.
If some $\boldsymbol\omega \in \Omega ^{d}(B;E^{{\ast }})$,
and if $\bold{\tilde{\eufb d}}$ is a cross-section of
the vector bundle $\delta ^{-1}E$, then, by definition,
$$\left<{\bold{\tilde{\eufb d}}},\delta^\star\boldsymbol\omega\right>({\eurb u}_{1},\ldots,{\eurb u}_{d})
=\left<\bold{\tilde{\eufb d}}(z),\delta^{\star}\boldsymbol\omega({\eurb u}_{1},\ldots,{\eurb u}_{d})\right>,
\quad\text{where}\quad{\eurb u}_{1},\ldots ,{\eurb u}_{d}\in T_{z}Z\,.
$$

Set $F=\wedge T^{{\ast }}$Z.  By means of the imbedding
${{\mathsf{id}}}\otimes \wedge \pi ^{{\ast }}$  the cross-sections of the bundle
$E^{{\ast }}\otimes \pi ^{-1}\wedge T^{{\ast }}Z$  are identified
with horizontal (with respect to $\pi $) $E^{{\ast }}$-valued
dif\-fe\-ren\-tial forms on the manifold $B$. The ${\frak F}_{\sssize B}$-module of these forms will be
denoted by $\Omega
_{\sssize B}(Z;E^{{\ast }})$.  The identification of it with a submodule in
$\Omega (B;E^{{\ast }})$  is carried out by the monomorphism $\pi
^{\sssize\#}$.

Let both
$\delta $ and $\delta _{t}$  be cross-sections of the fibred
manifold $B$ over $Z$, i.e.
$\pi \circ \delta _{t}=\pi \circ \delta ={{\fit {id}}}$.
In this case the vector $\frak b(z)$,
which is tangent to the curve
$\sigma _{{\delta }(z)}(t)=\delta _{t}(z)$,  will be vertical.

Set $E=VB$. As long as $\delta $ is a cross-section of the
projection $\pi $, the reciprocal image
$\delta ^{-1}{\bold d}_{\pi }\boldsymbol\beta$
of the cross-section ${\bold d}_{\pi }\boldsymbol\beta $  with respect
to the mapping $\delta $ will be a cross-section of the bundle
$\delta ^{-1}(VB)^{{\ast }}\otimes F$, since the bundle
$\delta ^{-1}\pi ^{-1}F\approx (\pi \circ \delta )^{-1}F$
has to be identified with $F$ by the projection
$\zeta ^{-1}{{\fit {id}}}$  onto the second factor,
${{\fit {id}}}^{-1}F\ni (z;{\eurb f})\mapsto{\eurb f}\in F$.
Let
$(\pi^{-1}\zeta)^{-1}\delta:\delta^{-1}\pi^{-1}F\rightarrow\pi^{-1}F$
be the projection onto the
second factor. It is straightforward that
$(\zeta ^{-1}\pi )\circ (\pi ^{-1}\zeta )^{-1}\delta
=\zeta ^{-1}(\pi \circ \delta )=\zeta ^{-1}{{\fit {id}}}$ .
Let now
$\delta ^{-1}{\vphantom{_x}\botsmash{\underset{\dsize\char"7E}{\botsmash{\beta}}}}^{\sssize V}$
denote the reciprocal image of the homomorphism
${\vphantom{_x}\botsmash{\underset{\dsize\char"7E}{\botsmash{\beta}}}}^{\sssize V}$.
Consider a cross-section $\bold{\tilde{\eufb b}}$ of the bundle
$\delta ^{-1}VB$  and denote $\frak b:Z\rightarrow VB$ the corresponding morphism
along the mapping $\delta $.
The definition of the contraction
$\left<\bold{\tilde{\eufb b}},\delta ^{-1}{\bold d}_{\pi }\boldsymbol\beta \right>\in \Gamma \{F\}$
is obvious (see the diagram of Fig.\mms\ref{Figure4} below):
$$\left<\bold{\tilde{\eufb b}},\delta ^{-1}{\bold d}_{\pi
}\boldsymbol\beta \right> =\zeta^{-1}{{\fit {id}}}\circ\delta
^{-1}{\vphantom{_x}\botsmash{\underset
{\dsize\char"7E}{\botsmash{\beta}}}}^{\sssize V}\circ \bold{\tilde{\eufb b}}
= \zeta ^{-1}\pi \circ (\pi ^{-1}\zeta )^{-1}\delta \circ \delta
^{-1}{\vphantom{_x}\botsmash{\underset
{\dsize\char"7E}{\botsmash{\beta}}}}^{\sssize V}\circ \bold{\tilde{\eufb b}}
=\zeta ^{-1}\pi \circ {\vphantom{_x}\botsmash{\underset{\dsize\char"7E}{\botsmash{\beta}}}}^{\sssize V}\circ
\frak b\,.$$
\begin{figure}[h]
\def\b{\vphantom{_x}\botsmash{\underset{\dsize\char"7E}{\botsmash{\beta}}}}
\def\s{\scriptscriptstyle}
$$\xy
\xymatrix{
&\delta^{\s-1}\pi^{\s-1}F
\ar`u+<0em,5ex>`[rrrrr]^{\zeta^{-1}({\fit {id}})}[rrrrr]
\ar[rrr]|{\object+{\scriptstyle(\pi^{-1}\zeta)^{-1}\delta}}
\ar[ddl]|<>(.507){\object+{}}
&&&\pi^{\s-1}F
\ar[rr]|{\object+{\scriptstyle\zeta^{-1}\pi}}
\ar[ddl]^{\pi^{-1}\zeta}
&&F
\ar[ddl]^{\zeta}\\
\delta^{\s-1}(VB)
\ar[ur]^{\delta^{-1}(\b^V)}
\ar[rrr]
&&&VB
\ar[ur]^{(\b^V)}
\ar[d]\\
Z
\ar[u]^{\bold{\tilde{\eufb b}}}
\ar@{.{>}}|{\object+{\scriptstyle\frak b}}[urrr]
\ar[rrr]|{\object+{\scriptstyle\delta}}
\ar`d[rrrrr]-<0em,5ex>`[rrrrr]_{{\fit {id}}}[rrrrr]
&&&B
\ar[rr]|{\object+{\scriptstyle\pi}}
&&Z
}\endxy$$
\caption{}
\label{Figure4}
\end{figure}

If we take $\wedge T^{{\ast }}Z$ in place of the fibre
bundle $F$ and if ${\eurb u}_{1},\ldots ,{\eurb u}_{d}\in T_{z}Z$,
then one can easily calculate:
\begin{multline*}
\left<\bold{\tilde{\eufb b}},\delta ^{-1}{\bold d}_{\pi }\boldsymbol\beta \right>
({\eurb u}_{1},\ldots ,{\eurb u}_{d})\\
=\bigl(\zeta ^{-1}\pi \circ {\vphantom{_x}\botsmash{\underset{\dsize\char"7E}{\botsmash{\beta}}}}^{\sssize V}\circ
\frak b(z)\bigr)({\eurb u}_{1},\ldots ,
{\eurb u}_{d})
=(d/dt)(\vphantom{_x}\botsmash{\underset{\dsize\char"7E}{\botsmash{\beta}}}\circ \sigma _{{\delta
}(z)}) ({\eurb u}_{1},\ldots ,{\eurb u}_{d})(0)\,.
\end{multline*}

Let
$(\pi ^{-1}\tau)^{-1}\delta
:\delta ^{-1}\pi ^{-1}TZ\rightarrow \pi ^{-1}TZ$
denote the standard projection
onto the second factor and let $\tau^{-1}{{\fit {id}}}$ denote the
obvious identification  ${{\fit {id}}}^{-1}TZ\approx TZ$  so that
$(\tau^{-1}\pi )\circ (\pi ^{-1}\tau)^{-1}\delta =\tau^{-1}{{\fit {id}}}$
(see again the diagram of Fig.\mms\ref{Figure4} with $\tau$ in place of $\zeta $ this time).
Let $\tau_{\sssize B}$ denote the projection $TB\rightarrow B$.
One computes (see Fig.\mms\ref{Figure5} below)
$T\pi \circ {\tau_{\sssize B}}^{-1}\delta =(\tau^{-1}\pi
)\circ \pi ^{\sssize T}\circ {\tau_{\sssize B}}^{-1}\delta
=(\tau^{-1}\pi)\circ(\pi^{-1}\tau)^{-1}\delta\circ(\delta^{-1}\pi^{\sssize T})
=\tau^{-1}{{\fit {id}}}\circ \delta ^{-1}\pi ^{\sssize T}$.
Composing with the mapping $\delta ^{\sssize T}$ it gives
$\tau^{-1}{{\fit {id}}}\circ (\delta ^{-1}\pi ^{\sssize T})\circ \delta
^{\sssize T} = T\pi \circ T\delta={{\mathsf{id}}}$; performing the transition
to the dual mappings, one obtains $\wedge \delta ^{\ast }\circ
(\delta ^{-1}\wedge \pi ^{{\ast }}) ={\overset \ast
{\tau}}{}^{-1}{{\fit {id}}}$.
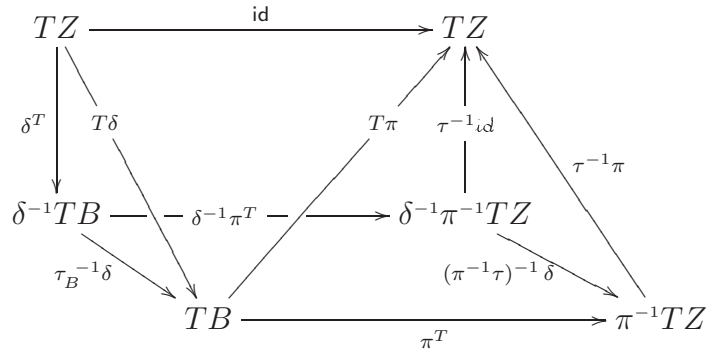
\begin{figure}[h]
\def\sz{\scriptstyle}
\def\ssz{\scriptscriptstyle}
$$\xy
\xymatrix{
TZ
\ar[rrr]^{{\mathsf{id}}}
\ar[dd]_{\delta^T}
\ar[dddr]|<>(.3){\object+{\sz T\delta}}
&&&TZ\\
\\
\delta^{\ssz-1}TB
\ar|<>(0.16){\object+{}}|<>(.41){\object+{\sz\delta^{-1}\pi^T}}|<>(0.67){\object+{}}[rrr]
\ar[dr]_(.35){\tau_{_B}{\!\!}^{-1}\delta}
&&&\delta^{\ssz-1}\pi^{\ssz-1}TZ
\ar[uu]|{\object+{\sz\tau^{-1}{\fit {id}}}}
\ar[dr]_(.35){(\pi^{-1}\tau){}^{-1}\,\delta}\\
&TB
\ar[uuurr]|<>(.7){\object+{\sz T\pi}}
\ar[rrr]_{\pi^T}
&&&\pi^{\ssz-1}TZ
\ar[uuul]_{\tau^{-1}\pi}
}\endxy$$
\caption{This is the upper part of the complete picture of Fig.\mms\ref{Figure13}.}
\label{Figure5}
\end{figure}

While the module of semi-basic differential forms
$\Gamma \{E^{{\ast }}\otimes \pi ^{-1}\wedge T^{{\ast }}Z\}$
is identified with the
module of horizontal differential forms on $B$ via the action of the
mapping
${{\mathsf{id}}}\otimes \wedge \pi ^{{\ast }}$
upon the corresponding
cross-sections, the reciprocal images with respect to $\delta $ are
identified by means of the action of the mapping
${{\mathsf{id}}}\otimes \delta ^{-1}\wedge \pi ^{{\ast }}$,
where
$\delta ^{-1}\wedge \pi ^{{\ast }}
:\delta ^{-1}\pi ^{-1}\wedge T^{{\ast }}Z\rightarrow \delta
^{-1}\wedge T^{{\ast }}B$
is the reciprocal image of the monomorphism
$\wedge \pi ^{{\ast }}$  with respect to the map $\delta $.
Thus,
if the differential form $\boldsymbol\alpha =\pi
^{\sssize\#}\boldsymbol\beta \equiv (\wedge \pi ^{{\ast }})\circ
\boldsymbol\beta $ is identified with the differential form $\boldsymbol\beta $
and
if the differential form
$\boldsymbol\omega =({{\mathsf{id}}}\otimes \wedge \pi ^{{\ast }})\circ {\bold d}_{\pi
}\boldsymbol\beta $ is identified with the differential form ${\bold d}_{\pi }\boldsymbol\beta $,
then the composition
$(\delta ^{-1}\wedge \pi ^{{\ast }})\circ \delta ^{-1}\boldsymbol\beta
=\delta ^{-1}\boldsymbol\alpha $
is identified with the differential form $\delta ^{-1}\boldsymbol\beta $
and the
composition $({{\mathsf{id}}}\otimes \delta ^{-1}\wedge \pi ^{{\ast }}) \circ
\delta ^{-1}{\bold d}_{\pi }\boldsymbol\beta =\delta ^{-1}\boldsymbol\omega $
is identified with the differential form
$\delta ^{-1}{\bold d}_{\pi }\boldsymbol\beta $.
(The diagrams of Figs\mms\ref{Figure14} and\mms\ref{Figure15} illustrate these identifications
and the accompanying  notational conventions as well as the computations following herein.)

The pulled-back differential form ${\delta_t}^{\star}\boldsymbol\alpha $
as a cross-section $Z\rightarrow\wedge T^{{\ast }}Z$ may be represented as follows:
\begin{align}
{\delta_t}^{\star}\boldsymbol\alpha
&={\delta_t}^{\sssize\#}{\delta_t}^{-1}\boldsymbol\alpha
 \equiv \wedge \delta ^{{\ast}}_{t}\circ {\delta_t}^{-1}\boldsymbol\alpha \notag \\
&\overset\sim\leftrightarrow \wedge \delta ^{{\ast }}_{t}\circ ({\delta_t}^{-1}
\wedge \pi ^{\ast })\circ {\delta_t}^{-1}\boldsymbol\beta
\;=\;{\overset \ast {\tau}}{}^{-1}{{\fit {id}}}\;\circ \;
{\delta_t}^{-1}\boldsymbol\beta\;\approx\;{\delta_t}^{-1}\boldsymbol\beta.
\label{8}\end{align}
The pulled-back form $\delta
^{\star}\boldsymbol\omega $  as a cross-section $Z\rightarrow \delta
^{-1}(VB)^{{\ast }} \otimes \wedge T^{{\ast }}Z$  may similarly be represented as
\begin{align*}
\delta ^{\star}\boldsymbol\omega
&={\delta}^{\sssize\#}{\delta}^{-1}\boldsymbol\omega
 \equiv ({{\mathsf{id}}}\otimes \wedge \delta
^{\ast })\circ \delta ^{-1}\boldsymbol\omega \\
&\overset\sim\leftrightarrow ({{\mathsf{id}}}\otimes \wedge \delta
^{\ast}) \circ ({{\mathsf{id}}}\otimes \delta^{-1}\wedge\pi^\ast)
\circ\delta^{-1}{\bold d}_{\pi}\boldsymbol\beta
=({{\mathsf{id}}}\otimes {\overset \ast {\tau}}{}^{-1}{{\fit {id}}})
\circ \delta^{-1}{\bold d}_{\pi }\boldsymbol\beta
\approx \delta ^{-1}{\bold d}_{\pi }\boldsymbol\beta.
\end{align*}
In (\ref{8}) we may also carry out explicitly the composition of
the map
${\overset \ast {\tau}}{}^{-1}{{\fit {id}}}$ with
${\delta_t}^{-1}\boldsymbol\beta $.  We insert the identity ${\overset
\ast {\tau}}{}^{-1}{{\fit {id}}}=({\overset \ast
{\tau}}{}^{-1}\pi)\circ(\pi^{-1}{\overset \ast
{\tau}})^{-1}\delta_{t}$ into (\ref{8}) and employ the definition of the
reciprocal image ${\delta_t}^{-1}\boldsymbol\beta $  together with the
definition of the $\pi $-morphism $\vphantom{_x}\botsmash{\underset{\dsize\char"7E}{\botsmash{\beta}}}$,
which read \;$({\overset \ast {\tau}}{}^{-1}\pi )\;\circ\;(\pi
^{-1}{\overset \ast {\tau}})^{-1} \delta _{t}\;\circ\;
{\delta_t}^{-1}\boldsymbol\beta\;=\;\vphantom{_x}\botsmash{\underset
{\dsize\char"7E}{\botsmash{\beta}}}\circ\delta _{t}$\,,\; to
obtain simply
${\delta_t}^{\star}\boldsymbol\alpha=\vphantom{_x}\botsmash{\underset
{\dsize\char"7E}{\botsmash{\beta}}}\circ\delta_{t}$.

\begin{rem}
{\rm We have in fact proved the following assertion.}
If $\boldsymbol\beta \in \Omega_{\sssize B}(Z)$  is a semi-basic differential form
with respect to a fibration $\pi :B\rightarrow Z$  and if
$\delta :Z\rightarrow B$
is a cross-section of that fibration, then
\begin{equation}
\boxed{\delta ^{\star}\boldsymbol\beta
=\vphantom{_{_{_{_x}}}}\botsmash{\underset{\dsize\char"7E}{\beta}}\circ
\delta }
\label{9}
\end{equation}
\end{rem}

Now it follows easily that
${\bold L}(\frak b)\boldsymbol\alpha =\left<{\bold{\tilde{\eufb b}}},{\delta}^{\star}\boldsymbol\omega\right>$.
\begin{proof}
Indeed, on the right-hand side here we have
\begin{align*}
\left<\bold{\tilde{\eufb b}},\delta^{\star}\boldsymbol\omega \right>({\eurb u}_{1},\ldots,{\eurb u}_{d})
&=\left<\bold{\tilde{\eufb b}},\delta ^{-1}{\bold d}_{\pi }\boldsymbol\beta \right> ({\eurb
u}_{1},\ldots ,{\eurb u}_{d}) \\
&=(d/dt)(\vphantom{_x}\botsmash{\underset{\dsize\char"7E}{\botsmash{\beta}}}\circ
\sigma _{\delta(z)})({\eurb u}_{1},\ldots,{\eurb u}_{d})(0),
\end{align*}
whereas on the left-hand side we proceed as follows,
\begin{align*}
{\bold L}(\frak b)\boldsymbol\alpha
&=(d/dt)({\delta_t}^{\star}\boldsymbol\alpha)(0)
=(d/dt)(\,{\overset \ast {\tau}}{}^{-1}{{\fit {id}}}\,\circ
\,{\delta_t}^{-1}\boldsymbol\beta\, )\,(0) \\
&= (d/dt)(\,{\overset \ast
{\tau}}{}^{-1}\pi \,\circ\,(\pi ^{-1}{\overset \ast
{\tau}})^{-1} \delta _{t}\,\circ\, {\delta_t}^{-1}\boldsymbol\beta
\,)\,(0) =(d/dt)(\vphantom{_x}\botsmash{\underset
{\dsize\char"7E}{\botsmash{\beta}}}\circ \delta _{t})(0),
\end{align*}
so that by evaluating on the vectors ${\eurb u}_{1},\ldots ,{\eurb
u}_{d}$  one regains that same result,
\begin{align*}
({\bold L}(\frak b)\boldsymbol\alpha ) ({\eurb u}_{1},\ldots ,{\eurb u}_{d}) &=
(d/dt)\bigl(\vphantom{_x}\botsmash{\underset
{\dsize\char"7E}{\botsmash{\beta}}}\circ \delta _{t}(z)\bigr)
({\eurb u}_{1},\ldots ,{\eurb u}_{d})(0) \\
&=(d/dt)(\vphantom{_x}\botsmash{\underset{\dsize\char"7E}{\botsmash{\beta}}}\circ\sigma_{\delta(z)})
({\eurb u}_{1},\ldots ,{\eurb u}_{d})(0).\qedhere
\end{align*}
\end{proof}

Not indicating explicitly the above mentioned identification of the
differential forms $\boldsymbol\alpha $ and $\boldsymbol\omega $ with the cross-sections
$\boldsymbol\beta $ and ${\bold d}_{\pi }\boldsymbol\beta $, one can write
\begin{equation}
\boxed{{\bold L}(\frak b)\boldsymbol\beta
=\left<\bold{\tilde{\eufb b}},\delta ^{\star}{\bold d}_{\pi }\boldsymbol\beta \right>
}\label{10}
\end{equation}
\section{Lagrange structure and the first variation\label{Sec3}}
\fancyhead[CE,CO]{\slshape SECTION~\thesection: FIRST VARIATION}
\subsection{Jet bundle structure}
By a classical field we mean a cross-section $\upsilon
:Z\rightarrow Y$ of a fibred manifold $\pi :Y\rightarrow Z$
over the base $Z$ in the category $C^{\infty }$. The jets of order
$r$ of such sections, each denoted $j_{r}\upsilon $, constitute the manifold $Y_{r}$ which is called
the $r^{\text{th}}$-order jet prolongation of the manifold $Y$ and we
put $Y=Y_{0}$. Projections ${^{r}\pi _s}:Y_{s}\rightarrow
Y_{r}$  for $r<s$  and $\pi _{r}:Y_{r}\rightarrow Z$  all are
surjective submersions and commute, $\pi _{r}\circ {^{r}\pi _s}=\pi
_{s}$.
Let ${\frak F}_{r}$ stand for
the ring ${\frak F}_{\sssize Y_{r}}$ of $C^{\infty }$ functions over
the manifold $Y_{r}$.
Monomorphisms ${^{r}\pi _s}^{\star}:{\frak F}_{r}\rightarrow {\frak F}_{s}$
and ${\pi _s}^{\star}:{\frak F}_{\sssize Z}\rightarrow {\frak F}_{s}$
allow us to identify the rings ${\frak F}_{r}$ and ${\frak F}_{\sssize Z}$
with the subrings ${^{r}\pi _s}^{\star} {\frak F}_{r}$ and
${\pi _s}^{\star}{{\frak F}_{\sssize Z}}$ of the ring ${\frak F}_{s}$.

Given another fibred manifold $Y^\prime $ over the same base $Z$ and
a base-preserving morphism $\phi :Y\rightarrow Y^\prime $, the morphism
\begin{equation}
J_{r}\phi :j_{r}\upsilon (z)\rightarrow j_{r}(\phi
\circ \upsilon)(z)
\label{11}
\end{equation}
from the manifold $Y_{r}$ to the manifold
${Y^\prime }_{r}$ is called the $r^{\text{th}}$-order prolongation of
the morphism $\phi$~\cite{Pommaret1978}.

\subsection{The variation of the Action functional}
A Lagrangian is a semi-basic (with respect to $\pi )$
differential form $\boldsymbol\lambda $ of maximal degree, $\boldsymbol\lambda \in \Omega
^{p}_{r}(Z)$, $p=\dim Z$.  Suppose again that the manifold $Z$ is
compact. Let ${\Cal Y}_{r}$ denote the space of smooth ($C^{\infty }$) cross-sections
of $Y_{r}$. The differential form $\boldsymbol\lambda $,
thought of as a morphism
$\vphantom{_x}\botsmash{\underset{\dsize\char"7E}{\lambda}} :Y_{r}\rightarrow
\wedge\!^{p}T^\ast Z$
along the projection $\pi
_{r}:Y_{r}\rightarrow Z$, defines a nonlinear differential
operator ${\check{\lambda}}$ in the space ${\Cal Y}=\Gamma
\{Y\}$  as follows:
\begin{equation}
{\check{\lambda}} (\upsilon
)=(j_{r}\upsilon )^{\star}\boldsymbol\lambda
=\vphantom{_x}\botsmash{\underset{\dsize\char"7E}{\lambda}}
\circ j_{r}\upsilon.
\label{12}
\end{equation}
The Action functional
$S = {\int_{_{\ssize Z}}} (j_{r}\upsilon )^{\star}\boldsymbol\lambda $
splits into the composition of three mappings, i.e.
$$
\boxed{S\ =\ \smallint_{\sssize Z}\ \circ\
{^{\star}\boldsymbol\lambda}\ \circ\ j_{r} }
$$
In the above, $j_{r}$ means
the $r^{\text{th}}$-order prolongation operator
$$ j_{r}:\Cal Y\rightarrow
\Cal Y_{r},\qquad \upsilon \mapsto j_{r}\upsilon ;
$$
$^{\star}\boldsymbol\lambda $  maps the space ${\Cal Y}_{r}$ into the
space of cross-sections of the determinant bundle ${\wedge\!^p}{T^\ast Z}$,
$$ ^{\star}\boldsymbol\lambda :{\Cal Y}_{r}\rightarrow \Omega ^{p}(Z),
\quad \upsilon_{r}\mapsto {\upsilon _r}^{\star}\boldsymbol\lambda,\quad
\upsilon_{r}\in {\Cal Y}_{r};
$$
$\smallint_{\sssize Z}$ is a linear
functional on the Banach space $\Omega ^{p}(Z)$, $$
\smallint_{\sssize Z}:\Omega ^{p}(Z)\rightarrow {\Bbb R},\qquad
\boldsymbol\beta \mapsto \int_{_{\ssize Z}}\boldsymbol\beta .
$$

The Euler-Lagrange equations for an extremal cross-section $\upsilon $
arise as the condition upon the Fr\'echet derivative ${\bold D}S(\upsilon )$
at the point $\upsilon $ to be equal to zero.
According to the chain rule,
\begin{equation}
{\bold D}(S)(\upsilon ) = \bigl( {\bold D}
\smallint_{\sssize
Z} \bigr) \bigl((j_{r}\upsilon)^{\star}\boldsymbol\lambda\bigr) \cdot
(\,{\bold D}\,{^{\star}\boldsymbol\lambda}\,)(j_{r}\upsilon ) \cdot
({\bold D}{j_{r}})(\upsilon )\;.
\label{13}
\end{equation}
Since the functional
$\smallint_{\sssize Z}$ is linear, its derivative ${\bold D}
\smallint_{\sssize Z}$  equals $\smallint_{\sssize Z}$ regardless
of the point $(j_{r}\upsilon )^{\star}\boldsymbol\lambda $.

We pass now to the  computation
of ${\bold D}(S)(\upsilon)$ in strictly consistent and {\it formal\/} manner.
\subsection{Differential of the map~$^\star\boldsymbol\lambda $}
In the classical field theory the variations of the Action functional
are computed with respect to those variations of functions, which are
fields themselves, that means, which are cross-sections of the
corresponding fibred manifolds. Thus, in the notations of Section\ms\ref{Sec2}, the
mappings $\delta $ and $\delta _{t}$ due to be cross-sections of the fibred
manifold $B\to Z$. In this case, and assuming also that the differential
form $\boldsymbol\alpha$ is semi-basic, one can write, according to (\ref{10}),
$$
{\bold L}(\frak b)\boldsymbol\alpha
=\left<\bold{\tilde{\eufb b}},\delta ^{\star}{\bold d}_{\pi }\boldsymbol\alpha \right>
\;.$$

Applying this formula along with the formula (\ref{6}) to the operator $^{\star}\boldsymbol\lambda $ by putting
$B=Y_{r}$, $\boldsymbol\alpha =\boldsymbol\lambda$, $\delta =j_{r}\upsilon $, and
substituting $\bold{\tilde{\eufb b}}$ by some
$\pmb{\widetilde{\frak y_r}}\in \Gamma \{{\upsilon_r}^{-1}VY_{r}\}$,
we come up finally to the desired calculative formula
\begin{equation}
\boxed{{\bold D}(^{\star}\boldsymbol\lambda )(j_{r}\upsilon
)\,{\boldkey.}\,\frak y_r
=\left<\pmb{\widetilde{\frak y_r}},(j_{r}\upsilon )^{\star}{\bold d}_{\pi }\boldsymbol\lambda\right>
}\label{14}
\end{equation}

\subsection{The permutation of the partial differentiations (Schwarz lemma)}\label{Sec3Par4Permutation}
In the following two Paragraphs we reproduce for the sake of the subsequent quotation the well-known technical
trick of the exchange in the order of applying the operation of the infinitesimal
variation and that of partial differentiation.
The tangent space to the manifold $\Cal Y_{r}$ at the point
$\upsilon _{r}$ is the space of cross-sections of the fibre bundle
${\upsilon_r}^{-1}V_{r}$ (from here on we introduce the more
economical notation $V_{r}$ in place of $V(Y_{r}))$. The manifold
$V_{r}$ along with being fibred over the base $Y_{r}$ by means of the
surmersion $\tau_{r}:T(Y_{r})\supset V_{r}\rightarrow Y_{r}$
is also fibred over the base $Z$ by means of the surmersion
$\pi_{r}\circ \tau_{r}:V_{r}\rightarrow Z$;
every time the latter
is implied we shall use the notation $(V_{r})_{\sssize Z}$. Cross-sections of
the bundle ${\upsilon_r}^{-1}V_{r}$  are identified with those
cross-sections of the fibred manifold $(V_{r})_{\sssize Z}$  which project
onto the mapping $\upsilon _{r}$, the totality of them denoted as
$\Gamma _{\upsilon _{r}}\{(V_{r})_{\sssize Z}\}$.  By means of the application
$J_{r}(\tau_{\sssize Y}):J_{r}(V_{\sssize Z})\rightarrow J_{r}(Y)\equiv Y_{r}$
the manifold $J_{r}(V_{\sssize Z})$ while fibred over the base $Z$ appears to
carry another fibred structure over the base $Y_{r}$. Say $\frak y$
be a lift of the cross-section $\upsilon :Z\rightarrow Y$ to the
vertical bundle $V$, then the cross-section $j_{r}\frak y$ of
the fibred manifold $J_{r}(V_{\sssize Z})$  is projected onto the
cross-section $j_{r}\upsilon $  under the mapping
$J_{r}(\tau_{\sssize Y})$ (see Fig.\mms\ref{Figure7} of Appendix~\ref{Appendix1}).

The isomorphism \ ${{\mathsf{is}}}$ \ between the manifolds $V_{s}$ and $J_{s}(V_{\sssize Z})$
over the base $Y_{s}$ is obtained from the following procedure.
To a vector ${\sigma _{y_{s}}}^\prime \in V_{y_{s}}(Y_{s})$, tangent
at the point $y_{s}=j_{s}\upsilon (z_{0})\in Y_{s}$  to the curve
$\sigma _{y_{s}}:t\mapsto j_{s}\upsilon _{t}(z_{0})$, the jet
$j_{s}\frak y(z_{0})\in J_{s}(V_{\sssize Z})$ is put into correspondence
the lift $\frak y$ being defined by the family of $\upsilon _{t}$,
i.e. $\frak y(z)={\sigma _{\upsilon (z)}}^\prime $ where for each $z$
the curve $\sigma _{\upsilon (z)}:t\mapsto \upsilon _{t}(z)$  is
contained in the fibre $Y_{z}$ of the fibred manifold Y.

Conversely, given a jet $j_{s}\frak y(z_{0})$  of some lift
$\frak y:Z\rightarrow V$ along the cross-section
$\upsilon =\tau_{\sssize Y}\circ \frak y:Z\rightarrow Y$,
we construct for each vertical tangent vector
$\frak y(z)$  an integral curve
$\sigma _{\upsilon (z)}(t)$
and hence the family of cross-sections
$\upsilon _{t}:z\mapsto \sigma _{\upsilon (z)}(t)$.
Then under the mapping
\ ${{\mathsf{is}}}$ \ the vertical tangent vector $(j_{s}\upsilon _{t}(z_{0}))^\prime $
is sent to the $s^{\text{th}}$-order jet at $z_{0}$ of the lift
$z\mapsto {\sigma _{\upsilon (z)}}^\prime =\frak y(z)$.

This very isomorphism  \ ${{\mathsf{is}}}$ \ acts upon the cross-sections of the
corresponding fibred manifolds over the base $Z$: if
$\frak y_{r}\in \Gamma _{\upsilon _{r}}\{(V_{r})_{\sssize Z}\}$,
then  ${{\mathsf{is}}}_{\sssize\#}(\frak y_{r})
\equiv {{\mathsf{is}}}\circ \frak y_{r}\in \Gamma _{\upsilon_{r}}\{J_{r}(V_{\sssize Z})\}$,
where $\Gamma _{\upsilon _{r}}$ in the second
membership relation means that only those cross-sections of
$J_{r}(V_{\sssize Z})$  count, which project onto $\upsilon _{r}$ under the
application $J_{r}(\tau_{\sssize Y})$.

\subsection{The differential of $j_{r}$}
Now we are going to prove the legitimacy of the diagram of Fig.\mms\ref{Figure6}\,:
\begin{figure}[h]
\def\sz{\scriptstyle}
\def\s{\scriptscriptstyle}
$$\xy\xymatrix{
&T_{j_r\upsilon}({\Cal Y}_r)=\Gamma_{j_r\upsilon}\{(V_r)_{\s Z}\}
\ar^{{\mathsf{is}}_{\#}}+DR;[dr]\\
T_\upsilon({\Cal Y})=\Gamma_\upsilon\{V_{\s Z}\}
\ar^{{\bold D}(j_r)(\upsilon)}[ur]+DL
\ar|{\object+{\sz j_r}}[rr]
&&\Gamma_{j_r\upsilon}\{J_r(V_{\s Z})\}
}\endxy$$
\caption{}
\label{Figure6}
\end{figure}
\begin{proof}
The differential of the mapping $j_{r}$ takes a vector
${\gamma _{\upsilon }}^\prime =\frak y$,
tangent to the curve
$\gamma _{\upsilon }:t\mapsto \upsilon _{t}$
at the point
$\upsilon =\gamma _{\upsilon }(0)\in \Cal Y$,
over to the vector
$\frak y_{r}=(j_{r}\circ \gamma _{\upsilon })^\prime $,
tangent to the curve
$t\mapsto j_{r}\upsilon _{t}$  at the point
$j_{r}\upsilon \in \Cal Y_{r}$, hence
$({\bold D}j_{r})(\upsilon )
:\Gamma
_{\upsilon }\{V_{\sssize Z}\}\rightarrow
\Gamma _{j_{r}\upsilon }\{(V_{r})_{\sssize Z}\}$.
The cross-section
$\frak y_{r}:z\mapsto {\sigma _{j_{r}\upsilon (z)}}^\prime $
is mapped under  ${{\mathsf{is}}}_{\sssize\#}$  into the cross-section
${{\mathsf{is}}}_{\sssize\#}(\frak y_{r}):z\mapsto j_{r}\frak y(z)$
of the fibred manifold $J_{r}(V_{\sssize Z})$, that is,
\ ${{\mathsf{is}}}_{\sssize\#}(\frak y_{r})=j_{r}\frak y$\,.
\,Thus in order to compute the Fr\'echet differential of the jet prolongation operator $j_{r}$
one may utilize the following permutation formula,
\begin{equation}
{\bold D}(j_{r})(\upsilon)\,{\boldkey.}\,{\frak y}=
{{\mathsf{is}}}^{-1}_{\sssize\#}j_{r}\frak y\;.
\label{15}
\end{equation}
\end{proof}

\subsection{The first variation}
From (\ref{14}) and (\ref{15}) we obtain the differential of the composed
mapping, $\,^{\star}\boldsymbol\lambda\,\circ\,j_{r}\,$,
\begin{align}
{\bold D}\,(\,^{\star}\boldsymbol\lambda\,\circ\,j_{r}\,)\,(\upsilon)\,{\boldkey.}\,{\frak y} & =
(\,{\bold D}\,{^{\star}\boldsymbol\lambda}\,)(j_{r}\upsilon ) \cdot
({\bold D}{j_{r}})(\upsilon )\,{\boldkey.}\,{\frak y} \notag\\
&=\left<\,\pmb{\boldkey(}{{\mathsf{is}}}^{-1}_{\sssize\#}j_{r}\frak y\pmb{\boldkey)\,\bold{\tilde{}}}{}\,,
\,(j_{r}\upsilon )^{\star}{\bold d}_{\pi }\boldsymbol\lambda \,\right>,
\label{16}
\end{align}

and, finally, from (\ref{13}),
the desired expression for the differential of the Action functional
$$
\boxed{{\bold D}S(\upsilon )\,{\boldkey.}\,{\frak y}
={\int_{_{\ssize Z}}}\left<\,\pmb{\boldkey(}{{\mathsf{is}}}^{-1}_{\sssize\#}j_{r}\frak y\pmb{\boldkey)\,\bold{\tilde{}}}{}\,,
\,(j_{r}\upsilon )^{\star}{\bold d}_{\pi }\boldsymbol\lambda \,\right>
}$$
\clearpage
\section{Integration by parts\label{Sec4}}
\fancyhead[CE,CO]{\slshape SECTION~\thesection: INTEGRATION BY PARTS}
\flushpar To proceed further we need to extend the definition of the fibre
differential ${\bold d}_{\pi }$ to the module of semi-basic
$\wedge\!^{d}V^{{\ast }}_{r}$-valued differential forms of
arbitrary degree $d$, \  $\Omega _{r}(Z;\wedge\!^{d}V^{{\ast
}}_{r})$,\    and to introduce the notion of total (global)
differential ${\bold d}_{t}$. This is being done in Appendix~\ref{Appendix2}.
Our considerations there as well as within this Section
essentially follow those of~\cite{Theses}.

\subsection{}
As far as we shall work with differential forms of different orders,
we shall frequently need to bring them together to the same base
manifold (if a differential form belongs to
$\Omega _{s}(Z;V^{{\ast }}_{r})$
we call the pair $(s,r)$ be the order of that form). We
recall that the homomorphism $(^{r}\pi _{s})^{\sssize V}$  maps $V_{s}$ to
${^{r}\pi_s}^{-1}V_{r}$. The dual homomorphism
$({^{r}\pi_s}\,^{\sssize V})^{{\ast}}$
acts upon the cross-sections of the dual vector
bundle ${^{r}\pi_s}^{-1}V^{{\ast }}_{r}$  through the composition
and hence it acts upon the module
$\Omega _{s}(Z;{^{r}\pi_s}^{-1}\wedge V^{{\ast }}_{r})$
by means of composing its elements
(viewed as cross-sections) with
$\wedge ({^{r}\pi_s}\,^{\sssize V})^{{\ast
}}\otimes {{\mathsf{id}}}$. Given a morphism ${{\fit g}}$ from whatsoever the
source be to the manifold $Y_{s}$, by $(^{r}\pi _{s})^{\sssize\#}$  the
reciprocal image of this action with respect to ${{\fit g}}$ will be denoted, namely, if $\boldsymbol\omega
\in \Gamma \{({^{r}\pi _s}\circ {{\fit g}})^{-1}\wedge V^{{\ast
}}_{r}\otimes (\pi _{s}\circ {{\fit g}})^{-1}\wedge T^{{\ast }}Z\}$, then
$(^{r}\pi _{s})^{\sssize\#}\doteqdot (V\,{^{r}\pi _s})^{\sssize\#}\doteqdot  \bigl({{\fit g}}^{-1}({^{r}\pi
_s}^{\sssize V})\bigr)^{\sssize\#}\doteqdot \bigl({{\fit g}}^{-1}(\wedge{^{r}\pi _s}^{\sssize
V}{}^{\ast})\bigr)_{\sssize\#} \doteqdot ({{\fit g}}^{-1}(\wedge{^{r}\pi _s}^{\sssize
V}{}^{\ast})\,\otimes \,{{\mathsf{id}}})_{\sssize\#}$, and $(^{r}\pi
_{s})^{\sssize\#}\boldsymbol\omega =({{\fit g}}^{-1}(\wedge{^{r}\pi _s}^{\sssize
V}{}^{\ast})\,\otimes \,{{\mathsf{id}}})\circ \boldsymbol\omega \in \Gamma \{{{\fit g}}^{-1}\wedge
V^{{\ast }}_{s}\otimes (\pi _{s}\circ {{\fit g}})^{-1}\wedge T^{{\ast }}Z\}$.
Let $\delta $ be another morphism which composes with ${{\fit g}}$ on the
left.  Then
\begin{equation}
\delta ^{\star}(^{r}\pi _{s})^{\sssize\#}\boldsymbol\omega
=(^{r}\pi _{s})^{\sssize\#}\delta ^{\star}\boldsymbol\omega.
\label{17}
\end{equation}

Indeed, by definitions, $\delta ^{\star}(^{r}\pi _{s})^{\sssize\#}\boldsymbol\omega =
({{\mathsf{id}}}\;\otimes\;\wedge \delta ^{{\ast }})\circ (\delta
^{-1}({^{r}\pi_s}^{\sssize\#}\boldsymbol\omega ))$.  But $\delta ^{-1}({^{r}\pi_s}
^{\sssize\#}\boldsymbol\omega )= \delta ^{-1}(({{\fit g}}^{-1}\wedge {^{r}\pi_s}^{\sssize V}{}^{\ast
}\;\otimes\;{{\mathsf{id}}})\circ \boldsymbol\omega )=(({{\fit g}}\circ \delta
)^{-1}\wedge {^{r}\pi_s}^{\sssize V}{}^{\ast }\;\otimes\;{{\mathsf{id}}})\circ
(\delta ^{-1}\boldsymbol\omega )$  and also $({{\mathsf{id}}}\;\otimes\;\wedge
\delta ^{\ast })\circ (({{\fit g}}\circ \delta )^{-1}\wedge
{^{r}\pi_s}^{\sssize V}{}^{\ast }\;\otimes\;{{\mathsf{id}}})=({{\fit g}}\circ
\delta )^{-1}\wedge {^{r}\pi_s} ^{\sssize V}{}^{\ast
}\;\otimes\;\wedge \delta ^{\ast }$.  On the other hand, $(^{r}\pi
_{s})^{\sssize\#}\delta ^{\star}\boldsymbol\omega = (({{\fit g}}\circ \delta
)^{-1}\wedge {^{r}\pi_s}^{\sssize V}{}^{\ast }\;\otimes\;{{\mathsf{id}}})\circ
({{\mathsf{id}}}\;\otimes\;\wedge \delta ^{\ast })\circ (\delta
^{-1}\boldsymbol\omega )$  and also $(({{\fit g}}\circ \delta )^{-1}\wedge {^{r}\pi_s}
^{\sssize V}{}^{\ast }\;\otimes\;{{\mathsf{id}}})\circ ({{\mathsf{id}}}\;\otimes\;
\wedge \delta ^{{\ast }})=({{\fit g}}\circ \delta )^{-1}\wedge {^{r}\pi_s}
^{\sssize V}{}^{\ast }\;\otimes\;\wedge \delta ^{\ast }$. So one concludes
that (\ref{17}) holds, q.e.d.

\subsection{}
We are now ready to write down the decomposition formula of {\smc kol\'{a}\v{r}}\kern0pt\footnotemark[2]
\footnotetext[2]{\kern3pt This is the generalization to an arbitrary
order of the decomposition formula adduced by {\smc trautman} in~\cite{Trautman1975}.  It has its counterpart in the algebra $\Omega (Y_{r})$,
where it is known under the name of the first variation formula~\cite{Krupka1973}. As long as the field theory is concerned and thereby the
splitting of the set of variables into independent and dependent ones
by $\pi $ is recognized, it is our opinion that the bigraded algebra
$\Omega_{s}(Z;V^{{\ast }}_{r})$
is a more appropriate object to work
with than the complete skew-symmetric algebra $\Omega (Y_{r})$.}
in terms of semi-basic differential forms
which take values in vector bundles $V^{{\ast }}$ and $V^{{\ast
}}_{r-1}$.

\begin{prop}\label{Proposition 1}
Given a Lagrangian $\boldsymbol\lambda \in \Omega^{p}_{r}(Z)$  there exist
semi-basic differential forms
$\boldsymbol\epsilon\in \Omega ^{p}_{2r}(Z;V^{{\ast }})$  and
$\boldsymbol\kappa\in \Omega ^{p-1}_{2r-1}(Z;V^{{\ast }}_{r-1})$ such
that
\begin{equation}
{^{r}\pi _{2r}}^{\star}\,{\bold d}_{\pi }\boldsymbol\lambda
\,=\,(^{0}\pi
_r)^{\sssize\#}\,\boldsymbol\epsilon+{\bold d}_{t}\,\boldsymbol\kappa. \label{18}
\end{equation}
The form $\boldsymbol\epsilon$ is unique inasmuch as its order is
fixed and equals $(2r,0)$~\cite{KolarNoveMesto}.
\end{prop}
\begin{proof}
The proof of Proposition~\ref{Proposition 1} may be carried out in the explicit local coordinate form by the undetermined coefficients
method. Non-existence of a formally intrinsic proof is closely related
to non-uniqueness of the differential form $\boldsymbol\kappa$.
\end{proof}
\begin{rem}\label{d_tExactness}
The differential form~$\boldsymbol\kappa$ is defined by the above decomposition up to a
${\bold d}_{t}$-closed term. Under certain additional restrictions on the structure of~$\boldsymbol\kappa$, it may be defined up to a
$d_t$-exact term, as seen from~\cite{Horak}.
\end{rem}

\subsection{}
By the Nonlinear Green Formula we mean herein the expression of the Fr\'echet derivative at the point
$\upsilon\in \Gamma \{Y\}$  of the operator
${\check{\lambda}}={^{\star}\boldsymbol\lambda}\circ j_{r}$
(see (\ref{12})) in terms of its transpose
$^{t}\bigl({\bold D}{\check{\lambda}}(\upsilon )\bigr)$
and of the Green operator ${\bold G}$~\cite{Bourbaki1971}
\begin{equation}
\left<({\bold D}{\check{\lambda}})(\upsilon )({\frak y}),1\right>
=\left<{\frak y},{}{^{t}\bigl({\bold D}{\check{\lambda}}(\upsilon)\bigr)}(1)\right>
+{\bold d}({\bold G}({\frak y},1)).
\label{19}
\end{equation}
We recall that the transpose operator
$^{t}\bigl({\bold D}{\check{\lambda}}(\upsilon )\bigr)$
is of the type
$(\wedge\!^{p}T^{{\ast }}Z)^{{\ast}}\otimes
\wedge\!^{p}T^{{\ast}}Z\rightarrow\upsilon^{-1}V^{{\ast}}
\otimes\wedge\!^{p}T^{{\ast }}Z$
whereas
${\bold D}{\check{\lambda}}(\upsilon )$
is of the type
$\upsilon ^{-1}V\rightarrow \wedge\!^{p}T^{{\ast }}Z$
and therefore the
Green operator has to be of the type
$(\upsilon ^{-1}V,\;(\wedge\!
^{p}T^{{\ast }}Z)^{{\ast }}\otimes \wedge\!^{p}T^{{\ast }}Z)
\rightarrow \wedge ^{p-1}T^{{\ast }}Z$.
Also the isomorphism
$(\wedge\!^{p}T^{{\ast }}Z)^{{\ast }}\otimes
\wedge\!^{p}T^{{\ast }}Z\approx {\Bbb R}_{\sssize Z}$
holds and under it the
contraction on the left-hand side of (\ref{19}) locally looks like
\begin{align*}
\left<\boldsymbol\mu,1\right>
&=\left<\mu _{0}{\bold d\xi }^{1}\wedge \ldots
\wedge {\bold d\xi }^{p},\partial /\partial \xi ^{1}\wedge \ldots
\wedge \partial /\partial \xi ^{p}\otimes
{\bold d}\xi^{1}\wedge \ldots \wedge {\bold d}\xi^{p}\right> \\
&= \mu _{0}{\bold d\xi }^{1}\wedge \ldots
\wedge{\bold d\xi }^{p}.
\end{align*}

That this Green formula (otherwise called the
 ``integration-by-parts'' formula) is obtained by so to say ``evaluating''
 the {\smc kol\'{a}\v{r}} decomposition formula (\ref{18}) along the
submanifold $j_{2r}\upsilon $, becomes clear to the end of present Section.
The demonstration will be carried out in three steps.

\subsubsection*{
(\romannumeral1)
}
Applying ${j_{2r}\upsilon}^{\star}$  to (\ref{18}) gives
\begin{equation}
(j_{r}\upsilon )^{\star}{\bold d}_{\pi }\boldsymbol\lambda
=(j_{2r}\upsilon )^{\star}({^{0}\pi _r})^{\sssize\#}\boldsymbol\epsilon
+(j_{2r}\upsilon )^{\star}{\bold d}_{t}\boldsymbol\kappa.
\label{20}
\end{equation}
On the other hand, by (\ref{16})
\begin{equation}
({\bold D}{\check{\lambda}})(\upsilon )({\frak y})
=\left<\pmb{\boldkey(}{{\mathsf{is}}}^{-1}_{\sssize\#}j_{r}{\frak y}\pmb{\boldkey)\,\bold{\tilde{}}},
(j_{r}\upsilon )^{\star}{\bold d}_{\pi }\boldsymbol\lambda \right>.
\label{21}
\end{equation}

\subsubsection*{
(\romannumeral2)
}
In what concerns the first addend of the right-hand side of (\ref{20}), one first
applies (\ref{17}) to get
$$
(j_{2r}\upsilon )^{\star}({^{0}\pi _r})^{\sssize\#}\boldsymbol\epsilon
=({^{0}\pi_r})^{\sssize\#}(j_{2r}\upsilon )^{\star}\boldsymbol\epsilon
$$
and  then consecutively (\ref{1}) and (\ref{A3}) together with (\ref{A1}) to arrive at
\begin{align}
\left<\pmb{\boldkey(}{{\mathsf{is}}}^{-1}_{\sssize\#}j_{r}{\frak y}\pmb{\boldkey)\,\bold{\tilde{}}},
({^{0}\pi _r})^{\sssize\#}(j_{2r}\upsilon )^{\star}\boldsymbol\epsilon\right>
&=\left<({^{0}\pi_r}^{\sssize V})_{\sssize\#}\pmb{\boldkey(}{{\mathsf{is}}}^{-1}_{\sssize\#}j_{r}
{\frak y}\pmb{\boldkey)\,\bold{\tilde{}}},(j_{2r}\upsilon )^{\star}\boldsymbol\epsilon\right> \notag\\
&=\left<\bold{\tilde{\eufb y}},(j_{2r}\upsilon )^{\star}\boldsymbol\epsilon\right>. \label{22}
\end{align}

\subsubsection*{
(\romannumeral3)
}
It remains to carry out some work upon the expression
$\left<\pmb{\boldkey(}{{\mathsf{is}}}^{-1}_{\sssize\#}j_{r}{\frak y}\pmb{\boldkey)\,\bold{\tilde{}}},
(j_{2r})^{\star}{\bold d}_{t}\boldsymbol\kappa\right>$.

\flushpar Suppose the vertical vector field ${\frak y}$ along
$\upsilon $ be extended to a vertical field ${\eufb v}$ on $Y$,
so that ${\frak y}={\eufb v}\circ \upsilon $.

As an intermediate step we first prove the following relationship:
\begin{lemma}\begin{equation}
\left<\pmb{\boldkey(}{{\mathsf{is}}}^{-1}_{\sssize\#}j_{r}{{\frak y}}\pmb{\boldkey)\,\bold{\tilde{}}},
(j_{2r}\upsilon )^{\star}{\bold d}_{t}\boldsymbol\kappa\right>
=(j_{2r}\upsilon )^{\star}\left<J_{r}({\eufb v}),
{\bold d}_{t}\boldsymbol\kappa\right>.  \label{23}
\end{equation}
\end{lemma}
\begin{proof}
By (\ref{A2}),
$\pmb{\boldkey(}{{\mathsf{is}}}^{-1}_{\sssize\#}j_{r}{\frak y}\pmb{\boldkey)\,\bold{\tilde{}}}
=\pmb{\boldkey(}J_{r}({\eufb v})\circ j_{r}\upsilon\pmb{\boldkey)\,\bold{\tilde{}}}
={j_{r}\upsilon }^{-1}(J_{r}({\eufb v}))$.
After the
definition of the pull-back,
$(j_{2r}\upsilon )^{\star}{\bold d}_{t}\boldsymbol\kappa
=(j_{2r}\upsilon )^{\sssize\#}{j_{2r}\upsilon }^{-1}{\bold d}_{t}\boldsymbol\kappa$.
So, on
the left-hand side of (\ref{23}) we come up to the expression
$$\left<{j_{r}\upsilon }^{-1}(J_{r}({\eufb v})),
(j_{2r}\upsilon)^{\sssize\#}{j_{2r}\upsilon }^{-1}{\bold d}_{t}\boldsymbol\kappa\right>\,.
$$
On the other hand, by (\ref{3}),
$${j_{2r}\upsilon }^{-1}\left<J_{r}({\eufb v}),{\bold d}_{t}\boldsymbol\kappa\right>
=\left<{j_{2r}\upsilon }^{-1}\,{^{r}\pi_{2r}}^{-1}\,J_{r}({\eufb v})\,,
\,{j_{2r}\upsilon }^{-1}{\bold d}_{t}\boldsymbol\kappa\right>\,.
$$
Next we apply to this the
$(j_{2r}\upsilon )^{\sssize\#}$ operation (in the only sensible way, i.e. with
respect to $Z$-variables in ${\bold d}_{t}\boldsymbol\kappa$; so it doesn't effect
the term ${j_{2r}\upsilon }^{-1}\,{^{r}\pi_{2r}}^{-1}\,J_{r}({\eufb v})$ of the
contraction) and obtain on the right-hand side of (\ref{23})
\begin{align*}
(j_{2r}\upsilon )^{\star}\left<J_{r}({\eufb v}),{\bold d}_{t}\boldsymbol\kappa\right>
&= \left<{j_{2r}\upsilon }^{-1}\,{^{r}\pi_{2r}}
^{-1}\,J_{r}({\eufb v})\,,\,(j_{2r}\upsilon )^{\sssize\#}{j_{2r}
\upsilon }^{-1}{\bold d}_{t}\boldsymbol\kappa\right> \\
&= \left<{j_{r}\upsilon }^{-1}J_{r}({\eufb v}),
(j_{2r}\upsilon )^{\sssize\#}{j_{2r}\upsilon }^{-1}{\bold d}_{t}\boldsymbol\kappa\right>.\qedhere
\end{align*}
\end{proof}

In (\ref{23}) we resort now consecutively to (\ref{A6}) and (\ref{A4}) to obtain the expected
formula,
\begin{equation}
\left<\pmb{\boldkey(}{{\mathsf{is}}}^{-1}_{\sssize\#}j_{r}{\frak y}\pmb{\boldkey)\,\bold{\tilde{}}},(j_{2r}\upsilon)^{\star}
{\bold d}_{t}\boldsymbol\kappa\right>={\bold d}({j_{2r-1}
\upsilon}^{\star}\left<J_{r-1}({\eufb v}),\boldsymbol\kappa\right>).  \label{24}
\end{equation}
\bigskip
Comparing (\ref{19}) with (\ref{18}) by means of (\ref{21}), (\ref{22}),
and (\ref{24}), and applying an analogue of (\ref{23}) with $\boldsymbol\kappa$
in place of ${\bold d}_{t}\boldsymbol\kappa$,
$$
\left<\pmb{\boldkey(}{{\mathsf{is}}}^{-1}_{\sssize\#}j_{r-1}{{\frak y}}\pmb{\boldkey)\,\bold{\tilde{}}},
(j_{2r-1}\upsilon )^{\star}\boldsymbol\kappa\right>
=(j_{2r-1}\upsilon )^{\star}\left<J_{r-1}({\eufb v}),
\boldsymbol\kappa\right>,
$$
the Reader
easily convinces Himself that
{\it
under the identification $j_{s}\frak y \overset{\tsize\sim}\leftrightarrow
\pmb{\boldkey(}{{\mathsf{is}}}^{-1}_{\sssize\#}j_{s}{{\frak y}}\pmb{\boldkey)\,\bold{\tilde{}}}\in j_s\upsilon^{-1}V_s$
in accordance with the isomorphism ${\mathsf{is}}^{-1}\:J_{s}(V_{\sssize Z}Y)\overset{\kern-2pt\tsize\approx}\to V(Y_{s})$\/}
the following assertion is true:

\begin{prop}
Let $\boldsymbol\lambda $ be an $r^{\text{th}}$-order Lagrangian for a
nonlinear field $\upsilon\in\Gamma\{Y\rightarrow Z\}$.
The variational derivative of the Action density
${\check{\lambda}}(\upsilon)=(j_{r}\upsilon )^{\star}\boldsymbol\lambda $
at $\upsilon $ is an
$r^{\text{th}}$-order differential operator
${\bold D}{\check{\lambda}}(\upsilon )$
in the space $\Gamma _{\upsilon }\{V_{\sssize Z}Y\}$  of the
variations of the field $\upsilon $. Let
$^{t}\bigl({\bold D}{\check{\lambda}}(\upsilon )\bigr)$
denote the transpose operator and let ${\bold G}$
denote the Green operator for ${\bold D}{\check{\lambda}}(\upsilon )$.
Then there exist semi-basic differential forms
$\boldsymbol\epsilon$ on $J_{2r}(Y)$ and
$\boldsymbol\kappa$ on $J_{2r-1}(Y)$ which take values in vector bundles
$(VY)^{\ast}$
and $\bigl(J_{r-1}(V_{\sssize Z}Y)\bigr)^{{\ast }}$ respectively (as
in Proposition~\ref{Proposition 1}) such that
\begin{align*}
({\bold D}{\check{\lambda}})(\upsilon )({\frak y})
&=\left<j_{r}{\frak y},
(j_{r}\upsilon )^{\star}{\bold d}_{\pi }\boldsymbol\lambda \right>; \\
^{t}\bigl({\bold D}{\check{\lambda}}(\upsilon )\bigr)\msp3(1)
&=(j_{2r}\upsilon )^{\star}\boldsymbol\epsilon\;; \\
{\bold G({\frak y})}\msp2(1)
&=\left<j_{r-1}{{\frak y}},(j_{2r-1}\upsilon )^{\star}\boldsymbol\kappa\right>.
\end{align*}

Whereas the Green operator is defined up to a ${\bold d}$-closed
term,
the differential form~$\boldsymbol\kappa$ is defined up to a
${\bold d}_{t}$-closed term, and, by Remark\ms\ref{d_tExactness}, may even be specified up to a
$d_t$-exact term.

The differential form $\boldsymbol\epsilon$ is defined uniquely and the
Euler-Lagrange equations arise as a local expression of the exterior
differential equation
\begin{equation}
\boxed{(j_{2r}\upsilon )^{\star}\boldsymbol\epsilon=0
}\label{25}
\end{equation}
\end{prop}

\bigskip
\subsection*{Discussion}
One would wish to introduce some intrinsically defined operator ${\bold E}$
to give an explicit expression to the Euler-Lagrange form $\boldsymbol\epsilon$
by means of
$$\boldsymbol\epsilon={\bold E}(\boldsymbol\lambda)\,.
$$
Considerable efforts were made, mainly by {\smc tulczyjew}~\cite{TulczyjewResolution}
and {\smc kol\'a\v r}~\cite{Kolar1977} in this direction which amount to defining
the operator ${\bold E}$ in terms of
some order-reducing derivations ${\boldsymbol\imath}^i$ of
degree $0$, acting in the exterior algebra of fibre differential forms
over the $r^{\text th}$-order prolongation manifold $Y_r$. These derivations
act as trivial ones in the ring of functions ${\frak F}_{r}$ over the manifold
$Y_r$ and are defined by prescribing their action upon one-forms in a local
chart $(\xi ^{i};\psi ^{a}_{\sssize\text{N}})$ as follows
$${\boldsymbol\imath}^i\bold d \psi ^{a}_{\sssize\text N}=
\begin{cases}
 \bold d \psi ^{a}_{{\sssize\text N}-\sssize1_i}, & \text{if $\text{\smc n}\doteqdot
(\nu_{1},\dots,\nu_{p})\geq 1_{i}$} \\
0, & \text{otherwise}\;.
\end{cases}
$$
We recall also the local expressions of partial total derivatives
${\bold D}_{i}={\bold D}_{t}(\partial /\partial \xi ^{i})$ (see Appendix~\ref{Appendix2}),
${\bold D}_{i}=\partial /\partial \xi ^{i}
+\psi ^{a}_{{\sssize\text{N}}+1_{i}}\partial/\partial \psi ^{a}_{\sssize\text{N}}$.
Let $\deg (\boldsymbol\varphi)=\text{degree of the fibre differential form $\boldsymbol\varphi$}$.
Neither ${\bold D}_{i}$ nor ${\boldsymbol\imath}^i$ have any intrinsic meaning,
but has the operator
\begin{equation}
\bold E=\deg\circ\,\bold d_\pi
+\sum_{\|{\sssize\text N}\|>0}
\frac {(-1)^{\|{\sssize\text N}\|}}{\text{\smc n}!}\, \bold
D_{\sssize\text N}{\boldsymbol\imath}^{\sssize\text N}\bold d_{\pi}\;.
\label{26}
\end{equation}
As far as this operator $\bold E$, defined initially in
$\Omega^{0}_{r}(Z;\wedge V^{\ast}_{r})$, acts trivially upon the subring $\frak F_{\sssize Z}$
of the ring $\frak F_{r}$, its action can be extended to the whole of
$\Omega_{r}(Z;\wedge V^{\ast}_{r})$  remaining trivial over the
subalgebra $\Omega(Z)$ of the algebra $\Omega_{r}(Z;\wedge V^{\ast}_{r})$.
Indeed, for a {\it parallelizable} $Z$ (which {\it locally}  is always true)
in course of considerations, similar to those of Paragraph\ms\ref{App2Par2TotalDiff} of Appendix~\ref{Appendix2}, we can
profit by {\it local} isomorphism
\,$\Omega_{r}(Z;\wedge V^{\ast}_{r})\approx
\Omega^{0}_{r}(Z;\wedge V^{\ast}_{r}) \otimes_{\frak F_{_Z}} \Omega(Z)$\,
to define
$$
\bold E(\boldsymbol\varphi\otimes\boldsymbol\mu)=\bold E(\boldsymbol\varphi)\otimes\boldsymbol\mu\,,
$$
whenever  $\boldsymbol\varphi\in \Omega^{0}(Z;\wedge V^{\ast}_{r})$ and $\boldsymbol\mu\in \Omega(Z)$.

Comparing $\bold E(\boldsymbol\lambda)$ in (\ref{26}) with (\ref{18}), it becomes
evident that the definition of the operator $\bold E$ amounts to the choice
of the form $\boldsymbol\kappa$. Since there is no natural way to make such choice
intrinsic, this seems to be the reason why the efforts to explicitly present a
consistently
intrinsic definition of the Euler-Lagrange operator $\bold E$ failed as far.
Of course, the problem melts down when passing to some quotient spaces of
differential forms. Immense development took place in that direction at the
level of cohomologies of bigraded complexes with the theory becoming still more
abstract and still more deviating from the original Euler-Lagrange expression.

We wish to emphasize, that it was due to the existence of the projection
$Y\to Z$ that the global splitting of variables into dependent and independent ones
became possible, which, in turn, led to the natural interpretation of the
term $\boldsymbol\epsilon$ in (\ref{18}) as a {\it semi-basic\/} differential form
{\it with values in a vector bundle\/}. In a more general framework,~--- that
of a contact
manifold in place of the jet prolongation $Y_{r}$ of a global surmersion $Y\to Z$,~---
such interpretation would never be possible. Even the Lagrangian itself could not be globally
presented as a semi-basic form with respect to independent variables, and thus
could not be specified in a canonical way~\cite{Krupka1973}. In practical computations,
however, one makes use of the local isomorphism between an $r^{\text{th}}$-order
contact manifold  $C^{p}_{r}(M)$ and the jet bundle $J_{r}(\Bbb R^{p};\Bbb R^{q})$,
$p+q=\dim M$, so even in this case it is possible to profit by the advantages
of the representation (\ref{25}). We now pass to the discussion of these advantages.

First, we see that the representation (\ref{25}) is natural. Indeed, the form
$\boldsymbol\epsilon$ originated from a Lagrangian, which ought to have been integrated
over the base manifold, hence ought to have manifested itself as a {\it differential
$p$-form\/}. But the variation was undertaken with respect to a vertical field
$\frak y$ and that vertical field enters as an argument to the linear transformation,
associated with
$\boldsymbol\epsilon$, which has nothing to do with the properties of $\boldsymbol\epsilon$
as a semi-basic $p$-form. So, $\boldsymbol\epsilon$ must manifest itself as a {\it vector
valued\/} $p$-form (more precisely~--- with values in the dual to the vector bundle
of infinitesimal variations).

Now, the expression (\ref{25}) is deprived of any other inessential parameters
which may have appeared during the process of variation. In particular,
no trace of any auxiliary vertical vector field (as in~\cite{Krupka1973})
remains. This allows representation of the solutions of the Euler-Lagrange
equations in the form of the integral manifolds of an exterior vector valued
differential system. Once recognized, such approach suggests
the framework of linear algebra: first, in investigating the symmetries
of the Euler-Lagrange equations, second, in solving the inverse problem of variational
calculus. In both cases the method consists in transforming the problem of
equivalence of two systems of differential equations,
one of them generated by the left-hand side of (\ref{25}),
into an algebraic problem of the equivalence of the
corresponding modules of vector differential forms using the Lagrange multiplies.
To apply this approach consistently, in the case of studying symmetries, one needs to generalize the notion of
the Lie derivative of a differential form to a derivative of a vector bundle
valued differential form (see \cites{Kolar1977,Munoz-Masque1985,Theses,Methods,SymVectorForms}).

\begin{appendix}
\section{Graded structure of vertical tangent bundles and prolongation of fiber transformations\label{Appendix1}}
\fancyhead[CE,CO]{\slshape APPENDIX~\thesection: PROLONGATION OF VERTICAL TANGENT BUNDLES}
\subsection{}
Within the notations of Section\ms\ref{Sec3} (Paragraph\ms\ref{Sec3Par4Permutation}), let
${^{r}{{\mathsf{pr}}}_s}:J_{s}(V_{\sssize Z})\rightarrow J_{r}(V_{\sssize Z})$
denote the projection
$j_{s}\frak y(z)\mapsto j_{r}\frak y(z)$.
Applying $J_{r}(\tau_{\sssize Y})$ to the target of the application
$^{r}{{\mathsf{pr}}}_s$, we get: $J_{r}(\tau_{\sssize
Y})\circ{^{r}{{\mathsf{pr}}}_s}~\bigl(j_{s}\frak y(z)\bigr) =(J_{r}\tau_{\sssize
Y})\bigl(j_{r}\frak y(z)\bigr) =j_{r}\upsilon (z)$. On the
other hand, ${^{r}\pi _s}\circ J_{s}(\tau_{\sssize
Y})~\bigl(j_{s}\frak y(z)\bigr) ={^{r}\pi _s}\bigl(j_{s}(\tau_{\sssize Y}\circ \frak y)(z)\bigr) ={^{r}\pi _s}\bigl(j_{s}\upsilon (z)\bigr)
=j_{r}\upsilon (z)$;  hence
${^{r}{{\mathsf{pr}}}_s}$ is fibred over $^{r}\pi _{s}$. As soon as
${^{r}{{\mathsf{pr}}}_s}$  is so fibred, it acts upon the cross-sections of the
reciprocal image ${j_{s}\upsilon }^{-1}J_{s}(V_{\sssize Z})$  of the fibred
manifold (in fact a vector bundle)
$J_{s}(\tau_{\sssize Y}):J_{s}(V_{\sssize Z})\rightarrow Y_{s}$.
After the general definitions, if we denote by
$\pmb{\boldkey(}j_{s}\frak y\pmb{\boldkey)\,\bold{\tilde{}}}$
the cross-section corresponding to the morphism
$j_{s}\frak y$  along
$j_{s}\upsilon $, and  by $({^{r}{{\mathsf{pr}}}_s})_{\sssize Y_{s}}$  the reduction of
${^{r}{{\mathsf{pr}}}_s}$  to the base $Y_{s}$, then
$({^{r}{{\mathsf{pr}}}_s})_{\sssize\#}\pmb{\boldkey(}j_{s}\frak y\pmb{\boldkey)\,\bold{\tilde{}}}
\in \Gamma \{j_{r}\upsilon ^{-1}J_{r}(V_{\sssize Z})\}$,
and
$({^{r}{{\mathsf{pr}}}_s})_{\sssize\#}\pmb{\boldkey(}j_{s}\frak y\pmb{\boldkey)\,\bold{\tilde{}}}
\doteqdot (j_{s}\upsilon
^{-1}({^{r}{{\mathsf{pr}}}_s})_{\sssize Y_{s}})\,\circ\,\pmb{\boldkey(}j_{s}\frak y\pmb{\boldkey)\,\bold{\tilde{}}}
=\pmb{\boldkey(}{^{r}{{\mathsf{pr}}}_s}\circ j_{s}\frak y\pmb{\boldkey)\,\bold{\tilde{}}}
=\pmb{\boldkey(}j_{r}\frak y\pmb{\boldkey)\,\bold{\tilde{}}}$. If abandon the tilde in the
superscripts, the projection $({^{r}{{\mathsf{pr}}}_s})_{\sssize\#}$ will obtain a
slightly different meaning as one acting from $\Gamma
\{J_{s}(V_{\sssize Z})\}$ into $\Gamma \{J_{r}(V_{\sssize Z})\}$ ``over $Z$'',
\begin{equation}
\boxed{({^{r}{{\mathsf{pr}}}_s})_{\sssize\#}(j_{s}{\frak y})=j_{r}{\frak y}
}\label{A1}
\end{equation}
\subsection{}
The isomorphism between the manifolds $V_s$ and $J_{s}(V_{\sssize Z})$
allows us to write down a useful
relationship between the jet of a restricted vertical vector field
${\eufb v}$, viewed as a cross-section of $V_{\sssize Z}$, and the prolongation
$J_{r}({\eufb v})$ of this field, obtained by prolonging its
one-parametric local group
as follows.

Consider a field ${\eufb v}\in \frak V$ of vertical tangent vectors,
generated by its local group $e^{t{\eufb v}}$. The
$r^{\text{th}}$-order prolongation $J_{r}({\eufb v})$  of ${\eufb v}$ is the
vector field generated by the local group $J_{r}(e^{t{\eufb v}}):
y_{r}\mapsto j_{r}(e^{t{\eufb v}}\circ \upsilon )(z)$, if
$y_{r}=j_{r}\upsilon (z)$. Consider thereto a restriction  ${\eufb v}\circ \upsilon $
of the vector field ${\eufb v}$ to some submanifold
$\upsilon (Z)$  of $Y$ ; it is an element $\frak y$ from $\Gamma_{\upsilon }\{V_{\sssize Z}\}$
(see Fig.\mms\ref{Figure7}):
\begin{figure}[h]
\def\ssz{\scriptscriptstyle}
\def\sz{\scriptstyle}
$$\xy
\xymatrix{
&\vrule width5em  height0ex
&J_r(V_{\ssz Z})
\ar[rr]^-{{\mathsf{is}}^{-1}}
\ar[drrr]|{\object+{\sz ^0{\mathsf{pr}}_r}}
&&V_r
\ar[dr]^-{V(^0\!\pi_r)}\\
&&&&&V\\
&&&\save\POS+<2em,-2ex>\drop+{Y_r}
    \ar[uul];_{J_r(\tau_{_Y})}|<>(0.763){\object+{}}
    \ar@{{}.>}|<>(0.327){\object+{}}|<>(.54){\object+{\sz J_r\eufb v}}|<>(0.744){\object+{}}[uur]
    \ar[drr]^{^0\!\pi_r}
    \POS="comment"\restore\\
Z
\ar@{{}.>}|{\object+{\sz j_r\frak y}}[uuurr]
\ar[uurrrrr]|{\object+{\sz \frak y=\eufb v\circ\upsilon}}
\ar@{{}.>}|{\object+{\sz j_r\upsilon}}"comment"
\ar@{{}.>}[rrrrr]|{\object+{\sz \upsilon}}
&&&&&Y
\ar@{{}.>}|{\object+{\sz\eufb v}}[uu]
}\endxy$$
\caption{}
\label{Figure7}
\end{figure}

Under the application  ${{\mathsf{is}}}^{-1}$ the point
$j_{r}({\eufb v}\circ \upsilon )(z)\in J_{r}(V_{\sssize Z})$
 transforms into the vertical tangent vector
$\bigl(j_{r}(e^{t{\eufb v}}\circ \upsilon )(z)\bigr)^\prime $,
which is nothing but exactly the value of the vertical field
$J_{r}({\eufb v})$  at the point $y_{r}$.  We conclude thereof that
the following formula (employed in~(\ref{23})) holds:
\begin{equation}
{{{\mathsf{is}}}}^{-1}\circ \,j_{r}({\eufb v}\circ \upsilon )
=J_{r}({\eufb v})\circ j_{r}\upsilon .
\label{A2}
\end{equation}
This relationship may be viewed as an alternative definition either of
$J_{r}({\eufb v})$  or of \ ${{\mathsf{is}}}$\,, \;as it is evidently
clear from Fig.\mms\ref{Figure7} again.
\subsection{}
Let $V(^{r}\pi _{s})$  denote the restriction of the tangent mapping
$T(^{r}\pi _{s})$  to the bundle $V_{s}$ of vertical  vectors tangent
to the fibred manifold $Y_{s}\rightarrow Z$. The mapping  \ ${{\mathsf{is}}}$ \ is
graded with respect to the pair of mappings, $V(^{r}\pi _s)$ and
$^{r}{{{\mathsf{pr}}}}_s$, over $^{r}\pi _{s}$ (see Fig.\mms\ref{Figure8})\,:

\begin{figure}[h]
\def\sz{\scriptstyle}
\def\s{\scriptscriptstyle}
$$\xy\xymatrix{
V(Y_s)
\ar|{\object+{\sz V(^r\!\pi_s)}}[rr]
\ar^{{\mathsf{is}}}[dr]
\ar_{\tau_s}[dd]
&&V(Y_r)
\ar|{\object+{}}[dd]
\ar^{{\mathsf{is}}}[dr]\\
&J_s(V_{\s Z})
\ar|<>(.24){\object+{\sz ^r{\mathsf{pr}}_s}}[rr]
\ar^{J_s(\tau_{_Y})}[dl]
&&J_r(V_{\s Z})
\ar^{J_r(\tau_{_Y})}[dl]\\
Y_s
\ar|{\object+{\sz^r\!\pi_s}}[rr]
&&Y_r
}\endxy$$
\caption{}
\label{Figure8}
\end{figure}
\begin{proof}
Under the tangent mapping $T(^{r}\pi _s)$ the vertical vector
${\sigma _{y_{s}}}^\prime $ projects onto the vector
${\sigma _{y_{r}}}^\prime $, tangent to the curve
$\sigma _{y_{r}}:t\mapsto j_{r}\upsilon_{t}(z_{0})$
at the point
$y_{r}={^{r}\pi _s}(y_{s})=j_{r}\upsilon (z_{0})$,
and that vector is identified
with the jet $j_{r}\frak y(z_{0})$  which is of course the image of
\;${{\mathsf{is}}}\,({\sigma _{y_{s}}}^\prime )$ \;under the projection
${^{r}{{{\mathsf{pr}}}}_s}$.
\end{proof}

Again, if we accept for a moment the
slight difference between $\Gamma _{\upsilon _{r}}\{J_{r}(V_{\sssize Z})\}$
and
\break
$\Gamma \{{\upsilon_r}^{-1}J_{r}(V_{\sssize Z})\}$, the mapping
${{\mathsf{is}}}^{-1}_{\sssize\#}$  will appear to act upon every
$\pmb{\boldkey(}j_{r}\frak y\pmb{\boldkey)\,\bold{\tilde{}}}\in \Gamma \{j_{r}\upsilon ^{-1}J_{r}(V_{\sssize Z})\}$.
Let $(^{r}\pi _{s})^{\sssize V}$  stand
for the reduction of $V(^{r}\pi _{s})$  to the base $Y_{s}$ by means
of the reciprocal image functor ${^{r}\pi _{s}}^{-1}$. According to
the general philosophy, we denote by $({^{r}\pi_s}\,^{\sssize
V})_{\sssize\#}$ its action upon the cross-sections of the bundle
${\upsilon_s}^{-1}V_{s}$ consisting in composing them with
${\upsilon_s}^{-1}({^{r}\pi _{s}}^{\sssize V})$,
$$ \boxed{({^{r}\pi _s}^{\sssize
V})_{\sssize\#} :\Gamma \{{\upsilon_s}^{-1}(V_{s})\}\rightarrow \Gamma
\{({^{r}\pi _s}\circ \upsilon _{s})^{-1}(V_{r})\} }
$$
If $\pmb{\widetilde{\frak y_s}}\in \Gamma
\{{\upsilon_s}^{-1}(V_{s})\}$ then $({^{r}\pi _s}^{\sssize
V})_{\sssize\#}\pmb{\widetilde{\frak y_s}} =\pmb{\boldkey(}{^{r}\pi _s}^{\sssize
V}\circ \frak y_{s}\pmb{\boldkey)\,\bold{\tilde{}}} =\pmb{\boldkey(}V(^{r}\pi _s)\circ \frak y_{s}\pmb{\boldkey)\,\bold{\tilde{}}}$.
Because of $V(^{r}\pi _{s})\circ
{{\mathsf{is}}}^{-1}={{\mathsf{is}}}^{-1}\circ {^{r}{{{\mathsf{pr}}}}_s}$  we
have
\begin{equation}
\boxed{({^{r}\pi _s}^{\sssize
V})_{\sssize\#}{{{\mathsf{is}}}}^{-1}_{\sssize\#}={{{\mathsf{is}}}}^{-1}_{\sssize\#}({^{r}{{{\mathsf{pr}}}}_s})_{\sssize\#}
}\label{A3}
\end{equation}

\section{Derivations on jet bundles\label{Appendix2}}
\fancyhead[CE,CO]{\iffloatpage{\slshape ILLUSTRATIONS}{\slshape APPENDIX~\thesection: DERIVATIONS ON JET BUNDLES}}
\flushpar In this Appendix we recall some very few preliminary properties of differentiation
technique in the graded modules over fibred manifolds for the sake of comprehension
and also to support several references encountered here and there in the text.
An interested Reader may appreciate at least three equivalent but conceptually
differing  definitions of the operator of total differential each revealing a
separate property quoted elsewhere and still all three intrinsic.

\subsection{The fibre differential\label{App2Par1FibreDiff}}
Let as usual some vector bundle $F$ be fibred over the base $Z$ by means of
the projection $\zeta $, and consider a manifold $B$, fibred over $Z$
by means of the surmersion $\pi $. In Section\ms\ref{Sec2} (formula~(\ref{7})) we have already
defined the fiber differential ${\bold d}_{\pi }\boldsymbol\beta \in \Gamma
\{(VB)^{{\ast }}\otimes \pi ^{-1}F\}$ of a cross-section $\boldsymbol\beta
\in \Gamma \{\pi ^{-1}F\}$. The Lie algebra ${\frak V}_{\sssize B}$ of
vertical vector fields on $B$ is a subalgebra in ${\frak X}_{\sssize B}$ and
hence acts as an algebra of derivations of the ring ${\frak F}_{\sssize B}$.
For a vertical vector field ${\eufb v}\in {\frak V}_{\sssize B}$  and a
function $f\in {\frak F}_{\sssize B}$  it is obvious that (cf.~the exact sequence (\ref{4}))
$$
\boxed{\left<{\eufb v},{\bold d}_{\pi }f\right>
=\left<\iota _{\sssize\#}{\eufb v},{\bold d}f\right>
}$$
For some fixed ${\eufb v}$ define an ${\Bbb R}$-endomorphism
${\bold D}_{\pi }({\eufb v})$  of the module $\Gamma \{\pi ^{-1}F\}$
by the rule
${\bold D}_{\pi }({\eufb v})\boldsymbol\beta
\doteqdot\left<{\eufb v},{\bold d}_{\pi }\boldsymbol\beta \right>$.
The map
${\bold D}_{\pi }:{\eufb v}\mapsto {\bold D}_{\pi }({\eufb v})$
is in fact a homomorphism
of modules over ${\frak F}_{\sssize B}$. It has the crucial property of
$$
{\bold D}_{\pi }({\eufb v})(f\cdot \boldsymbol\beta )
=\left<\iota _{\sssize\#}{\eufb v},
{\bold d}f\right>+f\cdot {\bold D}_{\pi }({\eufb v})\boldsymbol\beta
$$
and thus may be called the law of derivation of the
elements of the ${\frak F}_{\sssize B}$-module $\Gamma \{\pi ^{-1}F\}$  in
the direction of the elements of the algebra ${\frak V}_{\sssize B}$.
Exploiting this derivation law the differential ${\bold d}_{\pi }$ is
being extended to an exterior differentiation of the graded
${\frak F}_{\sssize B}$-module  ${\bold A} \bigl({\frak V}_{\sssize B};\Gamma \{\pi ^{-1}F\}\bigr)$
of exterior forms on ${\frak V}_{\sssize B}$ with values in $\Gamma \{\pi
^{-1}F\}$ by means of the commonly known rule~\cite{Koszul}
\begin{align*}
({\bold d}_{\pi }\boldsymbol\omega )({\eufb v}_{1},\ldots,{\eufb v}_{d+1})
&=\sum ^{d+1}_{{\sssize L}=1}(-1)^{{\sssize L}+1}{\bold D}_{\pi }({\eufb v}_{{\sssize L}})
\boldsymbol\omega ({\eufb v}_{1}, \ldots ,\widehat{{\eufb v}_{\sssize L}},\ldots,
{\eufb v}_{d+1}) \\
&+\sum _{{\sssize L}<{\sssize K}}(-1)^{{\sssize L}+{\sssize
K}}\boldsymbol\omega ([{\eufb v}_{{\sssize L}},
{\eufb v}_{{\sssize K}}],{\eufb v}_{1},\ldots , \widehat{{\eufb v}_{\sssize
L}},\ldots,\widehat{{\eufb v}_{\sssize K}},\ldots ,{\eufb v}_{d+1}).
\end{align*}

To obtain itself the actual operator we shall call the
fibre differential in the mo\-dule $\Omega_{\sssize B}(Z;\wedge {VB}^{\ast})$  set $F=\wedge T^{{\ast }}Z$
and employ the identification
${\bold A} \bigl({\frak V}_{\sssize B};\Gamma \{\pi ^{-1}F\}\bigr)\approx \Gamma
\{\wedge {VB}^{\ast}\otimes \pi ^{-1}F\}$.  Finally by
replacing $B$ with $Y_{r}$ one gets the initially desired operator
$$
\boxed{{\bold d}_{\pi }:\Omega ^{d}_{r}(Z;\wedge\!^{l}{V^{\ast}_r})\rightarrow \Omega ^{d}_{r}(Z;\wedge
^{l+1}{V^{\ast}_r})
}$$

\subsection{The total differential\label{App2Par2TotalDiff}}
A procedure, similar to that of the preceding Paragraph, will now be followed to define the operator of
the total dif\-fe\-ren\-tial
$$
\boxed{{\bold d}_{t}:\Omega ^{d}_{r}\bigl(Z;\wedge\!^{l}(V^{{\ast
}}_{r})\bigr)\rightarrow \Omega ^{d+1}_{r+1}\bigl(Z;\wedge\!^{l}(V^{{\ast
}}_{r+1})\bigr) }
$$

First a function $f\in {\frak F}_{r}$ will be treated
as one defining the fiber bundle homomorphism $(\pi
_{r},f):Y_{r}\rightarrow Z\times {\Bbb R}$  over the base $Z$
and we shall restrict the first-order prolongation (\ref{11}) of it,
$J_{1}(\pi _{r},f)$, to the manifold of holonomic jets $Y_{r+1}$
(see Fig.\mms\ref{Figure9} below).
\begin{figure}[h]
\def\subtld#1{\vphantom{_{\scriptscriptstyle\#}}\botsmash{#1{\!}_{_{_{\textstyle\char"7E}}}}}
$$\xy\xymatrix{
&T^\ast Z
\ar|{\object+{}}[dd]\\
Y_{r+1}
\ar@{^(->}[]+D-<0pt,1\jot>;[dd]
\ar^{\subtld{(\bold d_tf)}}[ur]
\ar[rr]
&&T^\ast Z\times\Bbb R\approx J_1(Z;\Bbb R)
\ar_{\text{pr}_1}[ul]
\ar@{{}={}}|{\object+{\approx}}[]!<2em,0em>;[dd]!<2em,0em>\\
&Z\\
J_1(Y_r)
\ar[ur]
\ar^-{J_1(\pi_r,f)}[rr]
&&(Z\times\Bbb R)_1\overset{\text{def}}=J_1(Z\times\Bbb R)
\ar[ul]
}\endxy$$
\caption{}
\label{Figure9}
\end{figure}
Then, the identification $(Z\times {\Bbb R})_{1}\approx
J_{1}(Z;{\Bbb R})\approx T^{{\ast }}Z\times {\Bbb R}$  will be
employed, where the cotangent bundle stands for the space of the
first-order jets $J_{1}(Z;{\Bbb R}_{0})$ with the target at $0\in
{\Bbb R}$. As the last step we apply the projection onto the first
factor and denote the entire construction by $f_{1}$.  Namely, if
$y_{r+1}= j_{r+1}\upsilon (z)\in Y_{r+1}$ and $y_{r}= {^{r}\pi
_{r+1}}(y_{r+1})$, then $f_{1}\bigl(j_{r+1}\upsilon (z)\bigr)= j_{1}\bigl(f\circ
j_{r}\upsilon -f(y_{r})\bigr)(z)$,  which is identified with the
differential form $({\bold d}_{t}f)(y_{r+1})$, such that
$({\bold d}_{t}f)(y_{r+1}){\boldkey.}\,{\eurb u}
={\bold d}(f\circ j_{r}\upsilon
)(z){\,\boldkey.}\,{\eurb u}$ whenever $\eurb u\in T_{z}Z$.
We denote by
${\bold d}_{t}f$ that cross-section of the vector bundle ${\pi_{r+1}}
^{-1}T^{{\ast }}Z$,  the corresponding $\pi _{r+1}$-morphism
$({\bold d}_{t}f)_{_{\dsize\char"7E}}$   of which coincides with
$f_{1}$ (see Fig.\mms\ref{Figure9} again),
$$
\boxed{\botsmash{({\bold d}_{t}f)_{_{_{_{\dsize\char"7E}}}}}={\text{pr}}_{1}\circ J_{1}(\pi
_{r},f) } $$
If one desired to restrict the hereby defined morphism
$\vphantom{_t}\botsmash{({\bold d}_{t}f)_{_{_{_{\dsize\char"7E}}}}}$ to
a cross-section $j_{r+1}\upsilon $ of the fibred manifold $Y_{r+1}$,
one would obtain $\vphantom{_t}\botsmash{({\bold d}_{t}f)_{_{_{_{\dsize\char"7E}}}}}\circ \,j_{r+1}\upsilon
=f_{1}\circ j_{r+1}\upsilon ={\bold d}(f\circ j_{r}\upsilon )\equiv
{\bold d}\bigl((j_{r}\upsilon )^{\star}f\bigr)$. But for the semi-basic
differential forms (\ref{9}) holds, so the following formula appears to
give a more ``working'' form of this definition,
however involving explicitly an arbitrary jet $j_{r+1}\upsilon $,
\begin{equation}
\boxed{(j_{r+1}\upsilon )^{\star}{\bold d}_{t}f={\bold d}((j_{r}\upsilon
)^{\star}f)
}\label{A4}
\end{equation}

We shall give an equivalent definition of this same operator of total
differential by employing the notion~\cite{Dhooghe1982} of the standard
horizontal lift ${{{\mathsf r}}}:{\pi_{r+1}}^{-1}(TZ)\rightarrow {^{r}\pi_{r+1}}
^{-1}\bigl(T(Y_{r})\bigr)$ (see the picture of Fig.\mms\ref{Figure10})\,:
\begin{figure}[h]
\def\s{\scriptscriptstyle}
\def\sz{\scriptstyle}
$$\xy
\xymatrix{
&&\save\POS-<1em,2.2ex>\drop+{\Bbb R}\POS="comment"\restore\\
\pi_{r\s+1}{\!\!}^{\s-1}\,TZ
\ar`u"comment""comment"|{\object+{\sz d_tf}}
\ar^-{{\mathsf r}}[rr]
\ar[dr]
&&^r{\!}\pi_{r\s+1}{\!\!}^{\s-1}\;T(Y_r)
\ar[rr]^-{\tau_r{\!}^{\s-1}\,(^r{\!}\pi_{r\s+1})}
\ar[dl]
&&T(Y_r)
\ar`u"comment""comment"|{\object+{\sz df}}
\ar[d]^{\tau_r} \\
&Y_{r+1}\ar[rrr]^{^r{\!}\pi_{r\s+1}}
&&&Y_r
}\endxy
$$
\caption{}
\label{Figure10}
\end{figure}

Assume vector
${\eurb u}$ be tangent to the curve $\sigma _{z}$ at the point $z\in
Z$.  The vector bundle homomorphism ${{{\mathsf r}}}$ takes the pair
$(y_{r+1},{\eurb u})$ over to the pair $\bigl(y_{r+1},(j_{r}\upsilon \circ
{\sigma _{z})}^\prime\bigr)$. Let $df$  denote the function on $TY_{r}$
which corresponds to the differential form  ${\bold d} f$. By
composing this $df$ with the standard projection
${\tau_r}^{-1}(^{r}\pi _{r+1}):{^{r}\pi_{r+1}}
^{-1}\bigl(T(Y_{r})\bigr)\rightarrow T(Y_{r})$  one can apply it to the lift
${{{\mathsf r}}}(y_{r+1},{\eurb u})$  and define
$$
(d_{t}f)(y_{r+1},{\eurb u})\;\doteqdot\;(df)\circ {\tau_r}^{-1}(^{r}\pi
_{r+1})\circ {{{\mathsf r}}}\ \;(y_{r+1},{\eurb u})\,.
$$
Or, introducing the notion of the dual
${{{\mathsf r}}}^{\sssize\#}=({{{\mathsf r}}}^{{\ast }})_{\sssize\#}$
of the mapping
${{{\mathsf r}}}_{\sssize\#}:{\frak H}_{r+1}(Z)\rightarrow {\frak H}_{r+1}(Y_{r})$,
this definition amounts to
$$
\boxed{{\bold d}_{t}f={{{\mathsf r}}}^{\sssize\#}\,\,{^{r}\pi_{r+1}}^{-1}\,\,{\bold df}
}$$

That the two definitions meet is obvious from the standard definition
of the exterior differential, namely
\begin{align*}
{\bold d}(f\circ j_{r}\upsilon )(z)\,{\boldkey.}\,{\eurb u}
&=(d/dt)(f\circ j_{r}\upsilon \circ\sigma _{z})(0) \\
&= {\bold d}f\bigl(j_{r}\upsilon (z)\bigr)\,{\boldkey.}\,
{\tau_r}^{-1}(^{r}\pi _{r+1})\,\bigl(\,{{{\mathsf r}}}\,(y_{r+1},{\eurb u})\,\bigr) \\
&= {\bold d}f(\pi _{r+1}y_{r+1})\,{\boldkey.}\,{\tau_r}^{-1}
(^{r}\pi _{r+1})\,\bigl(\,{{{\mathsf r}}}\,(y_{r+1},{\eurb u})\,\bigr) \\
&= (df)\circ \bigl({\tau_r}^{-1}(^{r}\pi_{r+1})\bigr)\ \,\bigl(\,{{{\mathsf r}}}\,(y_{r+1},{\eurb u})\,\bigr)\,,\quad \text{q.e.d.}
\end{align*}

By means of the lift ${{{\mathsf r}}}$ the module ${\frak X}_{\sssize Z}$ of vector fields
on $Z$ converts into an ${\frak F}_{r+1}$-module of de\-ri\-vations from
the commutative algebra ${\frak F}_{r}$ into the commutative algebra
${\frak F}_{r+1}$. To each field
${{\eufb z}}\in {\frak X}_{\sssize Z}$
there appears to be attached thus a derivation
${\bold D}_{t}({{\eufb z}})$
according to the rule
${\bold D}_{t}({{\eufb z}})f
\doteqdot({\bold d}_{t}f){\boldkey.}({\pi_{r+1}}^{-1}{{\eufb z}})$,
where the dot denotes
the contraction between $\Omega ^{1}_{r+1}(Z)$ and ${\frak H}_{r+1}(Z)$ calling to mind that
${\pi_{r+1}}^{-1}{\frak X}_{\sssize Z}\subset {\frak H}_{r+1}(Z)$.
This derivation ${\bold D}_{t}({{\eufb z}})$
is then extended to a derivation of degree~0
from the skew-symmetric graded algebra
$\Phi _{r}\doteqdot {\bold A}({\frak V}_{r})\approx \Gamma \{\wedge
V^{{\ast }}_{r}\}$
over ${\frak F}_{r}$ into the algebra $\Phi _{r+1}$ over
${\frak F}_{r+1}$
by the requirement that it commutes with ${\bold d}_{\pi }$~\cite{Kolar1977}. The so extended derivation law ${\bold D}_{t}({{\eufb z}})$
for each fixed ${{\eufb z}}$ is viewed as an ${\Bbb R}$-homomorphism from
the  ${\frak F}_{r}$-module $\Phi _{r}$ into the ${\frak
F}_{r+1}$-module $\Phi _{r+1}$. Since the map
${\bold D}_{t}:{{\eufb z}}\rightarrow {\bold D}_{t}({{\eufb z}})$
besides
that being itself a moduli homomorphism over the algebra homomorphism
${\pi_{r+1}}^{\star}:{\frak F}_{\sssize Z}\rightarrow {\frak F}_{r+1}$,
possesses also that crucial property of a derivation law,
$$
{\bold D}_{t}({{\eufb z}})(f\cdot \boldsymbol\varphi )
=({\bold D}_{t}({{\eufb z}})f)\cdot\, {^{r}\pi_{r+1}}^{\sssize\#}\boldsymbol\varphi\,
+\,({^{r}\pi_{r+1}}^{\star}f)\cdot {\bold D} _{t}({{\eufb z}})\boldsymbol\varphi\,,
$$
we can extend the total differential ${\bold d}_{t}$ to
an exterior differentiation operator from the  ${\frak F}_{r}$-module
${{\bold A}\msp1\negmedspace^{d}}{}({\frak X}_{\sssize Z};\Phi _{r})$  into the  ${\frak F}_{r+1}$-module
${{\bold A}\msp1\negmedspace^{d+1}}{}\msp1({\frak X}_{\sssize Z};\Phi _{r+1})$  in
a similar way as in Paragraph\ms\ref{App2Par1FibreDiff},
\begin{align*}
({\bold d}_{t}\boldsymbol\omega )({{\eufb z}}_{1},\ldots,{{\eufb z}}_{d+1})
&=\sum^{d+1}_{i=1}(-1)^{i+1}{\bold D}_{t}
({{\eufb z}}_{i})\boldsymbol\omega ({{\eufb z}}_{1},\ldots,{\widehat{{\eufb z}}}_{i},
\ldots,{{\eufb z}}_{d+1}) \\
&+\sum_{i<j}(-1)^{i+j}\boldsymbol\omega ([{{\eufb z}}_{i},{{\eufb z}}_{j}],
{{\eufb z}}_{1},\ldots,{\widehat{{\eufb z}}}_{i},
\ldots,{\widehat{{\eufb z}}}_{j},\ldots,{{\eufb z}}_{d+1}).
\end{align*}

The latter formula represents a local operator. We now transfer it to the module
${\bold A} ({\frak H}_{r}(Z);\Phi _{r})$ taking the advantage of the fact that
the module ${\frak X}_{\sssize Z}$ is locally free, i.e.
locally the base manifold $Z$ is parallelizable. For a {\it parallelizable\/} $Z$ the module
${\frak H}_{r}(Z)=\Gamma \{{\pi_r}^{-1}(TZ)\}$  is isomorphic to
the extension
${\frak F}_{r}\otimes _{{\frak F}_{_Z}}{\frak X}_{\sssize Z}$
of  ${\frak X}_{\sssize Z}$ obtained by the extension of the main
ring from  ${\frak F}_{\sssize Z}$ to  ${\frak F}_{r}$ relative to
${\pi_r}^{\star}$. Suppose the differential forms ${\bold d\xi }^{i}$ perform
some implementation of the parallelizability of $Z$, then
the isomorphism is given by ${{\eufb h}}\mapsto({\pi_r}^{-1}{\bold d\xi }^{i}){\boldkey.}
{{\eufb h}}\otimes (\partial /\partial \xi ^{i})$, and the inclusion
${\pi_r}^{-1}:{\frak X}_{\sssize Z}\rightarrow {\frak H}_{r}(Z)$
is represented under it by
${{\eufb z}}\mapsto 1\otimes {{\eufb z}}$.
Also the module
$\Omega^{d}_{r}(Z)=\Gamma \{{\pi_r}^{-1}{\wedge\!^{d}}T^\ast Z)\}$
is in this case isomorphic to the  extension
${\frak F}_{r}\otimes _{{\frak F}_{_Z}}{{\bold A}\negmedspace^{d}}{}\msp1({\frak X}_{\sssize Z})$.
There exists a homomorphism of modules from
${\frak F}_{r}\otimes _{\frak F_{_Z}}{{\bold A}\negmedspace^{d}}{}\msp1({\frak X}_{\sssize Z})$
to ${{\bold A}\negmedspace^{d}}{}\msp1({\frak F}_{r}
\otimes _{{\frak F}_{_Z}}{\frak X}_{\sssize Z})$
over ${\frak F}_{r}$ under which an element
$1\otimes\boldsymbol\mu\in{\frak F}_{r}\otimes _{\frak F_{_Z}}{{\bold A}\negmedspace^{d}}{}\msp1({\frak X}_{\sssize Z})$
goes over to the element
$\boldsymbol\beta\in {{\bold A}\negmedspace^{d}}{}\msp1({\frak F}_{r}
\otimes _{{\frak F}_{_Z}}{\frak X}_{\sssize Z})$,
defined by prescribing its values at the elements of
${\frak F}_{r}\otimes _{{\frak F}_{_Z}}{\frak X}_{\sssize Z}$
as
$\boldsymbol\beta (1\otimes {{\eufb z}}_{1},\ldots,1\otimes{{\eufb z}}_{d})
={\pi_r}^{\star}\bigl(\boldsymbol\mu({{\eufb z}}_{1},\ldots,{{\eufb z}}_{d})\bigr)$.
This homomorphism in case of finite-dimensional ${\frak X}_{\sssize Z}$
is in fact biunique and represents the isomorphism
$\Omega^{d}_{r}(Z)\approx {{\bold A}\negmedspace^{d}}{}\msp1(({\frak H}_{r}(Z))$ in terms of
the above extensions.
Multiplying it tensorwise by $\boldkey i\boldkey d:\Phi_{r}\rightarrow\Phi_{r}$,
and applying the identification
$\Phi _{r}\otimes _{{{\frak F} }_{_r}}{\frak F}_{r}
\otimes _{{\frak F}_{_Z}}{{\bold A}\negmedspace^{d}}{}\msp1({\frak X}_{\sssize Z})\approx
\Phi _{r}\otimes _{{\frak F}_{_Z}}{{\bold A}\negmedspace^{d}}{}\msp1({\frak X}_{\sssize Z})$,
it is possible to construct the
following commutative diagram of Fig.\mms\ref{Figure11}\,:
\enlargethispage{7truemm}
\begin{figure}[h]
\def\s{\scriptscriptstyle}
\def\sz{\scriptstyle}
$$\xy
\xymatrix{
\bold A(\frak F_{_{\sz r}}\otimes_{_{{\sz\frak F}_{_Z}}}\!\!\frak X_{_ Z};\,\Phi_{_{\sz r}})
\ar|-{\object+{\pmb{\boldkey(}\pi_r{\!}^{\s-1}\pmb{\boldkey)^{\boldsymbol\ast}}}}[rr]
&&\bold A(\frak X_{_Z};\Phi_{_{\sz r}})\\
\Phi_{_{\sz r}}\otimes_{_{{\sz\frak F}_{_r}}}\!\!\bold A(\frak F_{_{\sz r}}\otimes_{_{{\sz\frak F}_{_Z}}}\!\!\frak X_{_Z})
\ar[u]
&&\Phi_{_{\sz r}}\otimes_{_{{\sz\frak F}_{_Z}}}\!\!\bold A(\frak X_{_Z})
\ar[ll]
\ar[u]
}\endxy
$$
\caption{}
\label{Figure11}
\end{figure}

The vertical arrows are evident and in the case of finite-dimensional modules they
are in fact isomorphisms. For example, the one on the left maps an
element $\boldsymbol\varphi \otimes \boldsymbol\beta $ into the exterior form
$(1\otimes {{\eufb z}}_{1},\ldots,1\otimes {{\eufb z}}_{d})\mapsto
\boldsymbol\beta(1\otimes{{\eufb z}}_{1},\ldots,1\otimes{{\eufb z}}_{d})
{\boldkey.}\boldsymbol\varphi $. We conclude that
$\pmb{\boldkey(}{\pi_r}^{-1}\pmb{\boldkey)^{\boldsymbol\ast }}$ is an
isomorphism insofar and therefore the total differential ${\bold d}_{t}$
turns out to be defined as a local operator in the module
${\bold A} ({\frak H}_{r}(Z);\Phi _{r})$ too.
It possesses the property of a
derivation of degree~+1  of the graded module
$\Omega _{r}(Z;\wedge V^{{\ast }}_{r})$
over the graded algebra $\Omega _{r}(Z)$, i.e.
\begin{equation}
{\bold d}_{t}(\boldsymbol\beta \wedge \boldsymbol\omega )
={\bold d}_{t}\boldsymbol\beta
\wedge{^{r}\pi_{r+1}}^{\sssize\#}{^{r}\pi_{r+1}}^{\star}\boldsymbol\omega
+(-1)^{d}\,{^{r}\pi_{r+1}}^{\star}\boldsymbol\beta\wedge {\bold d}_{t}\boldsymbol\omega,
\label{A5}
\end{equation}
whenever $\boldsymbol\beta\in\Omega ^{d}_{r}(Z)$.

One more way to make operator $\bold d_{t}$ act over the whole of the graded
module $\bold A (\frak H_{r};\Phi_{r})$ is to first extend the derivation,
defined in $\frak F_{r}$ by
\begin{alignat*}{2}
\bold D_{t}(\eufb h)&\doteqdot (\bold d_{t}f){\,\boldkey.\,}\eufb h,&\qquad
\eufb h&\in \frak H_{r+1}(Z)\,,
\intertext{to the whole of $\Phi_{r}\approx \Omega^{0}_{r}(Z;\wedge V^{\ast}_{r})$
engaging the similar procedure as above, and then to extend the total differential
$\bold d_{t}:\Phi_{r}\to {\bold A\negmedspace^{1}}{}\msp1(\frak H_{r+1};\Phi_{r+1})$,
defined by}
(\bold d_{t}\boldsymbol\varphi){\,\boldkey.\,}\eufb h&\doteqdot\bold D_{t}(\eufb h)\msp1\boldsymbol\varphi,
&\qquad \eufb h&\in \frak H_{r+1}\,,
\end{alignat*}
to the module $\bold A(\frak H_{r};\Phi_{r})$ through the property
$\bold d_{t}{}^2=0$  together with the property~(\ref{A5}).\kern0pt\footnotemark[2]
\footnotetext[2]{\kern3pt The bigraded algebra
${\bold A} ({\frak H}_{r-1};
\Phi _{r-1})\approx \Omega _{r-1}(Z;\wedge V^{{\ast }}_{r-1})$
may be converted into exterior one by applying the dual of
the Cartan contact form
$T(Y_{r})\rightarrow {^{r-1}\pi_r}^{-1}(V_{r-1})$
with subsequent alternation. To the operators
 ${\bold d}_{\pi }$ and ${\bold d}_{t}$ defined in $\Omega
_{r-1}(Z;\wedge V^{{\ast }}_{r-1})$ correspond under this
 conversion the operators ${\bold d}_{\sssize V}$ and ${\bold d}_{\sssize H}$
of {\smc tulczyjew}~\cite{TulczyjewResolution} defined in $\Omega
(Y_{r})$.}

Let us compute the total differential of a contraction (see also~\cite{Horak}).
\begin{lemma}
Let  ${\eufb v}\in{\frak V}$,
$\boldsymbol\omega\in\Omega _{r}(Z;V^{\ast }_{r})$,
and let $J_{r}({\eufb v})$
denote the $r^{\text{th}}$-order prolongation of the vector field
${\eufb v}$ by means of prolonging its local one-parametric group.
The following formula holds
\begin{equation}
{\bold d}_{t}\left<J_{r}({\eufb v}),\boldsymbol\omega \right>
=\left<J_{r+1}({\eufb v}),{\bold d}_{t}\boldsymbol\omega \right>.  \label{A6}
\end{equation}
\end{lemma}
\begin{proof}[We give a brief proof]
Locally the algebra
$\Omega _{s}(Z;V^{{\ast }}_{s})$
is generated over $\Omega _{s}(Z)$ by
$\Omega ^{0}_{s}(Z;V^{{\ast }}_{s})$. Also
$\left<{\eufb v}_{s},\boldsymbol\beta \wedge \boldsymbol\varphi \right>
=\boldsymbol\beta\wedge\left<{\eufb v}_{s},\boldsymbol\varphi \right>$  for
$\boldsymbol\beta \in\Omega _{s}(Z)$, $\boldsymbol\varphi\in\Omega^{0}_{s}(Z;V^{{\ast}}_{s})$
and ${\eufb v}_{s}\in{\frak V}_{s}$,
so one may restrict oneself to the case
$\boldsymbol\omega ={\bold d}_{\pi }f$.
Let $\boldsymbol\beta ={\bold d}_{t}f$. Since ${\bold d}_{t}$
and ${\bold d}_{\pi }$ commute, ${\bold d}_{t}\boldsymbol\omega $
equals ${\bold d}_{\pi }\boldsymbol\beta $.
We compute:
\begin{align*}
\left<J_{r}({\eufb v}),{\bold d}_{\pi }f\right>(y_{r})
&=(Tf)\bigl(J_{r}({\eufb v})\bigr)(y_{r}) \\
&=(d/dt)f\bigl(J_{r}(e^{t{\eufb v}})(y_{r})\bigr)(0);
\end{align*}
\begin{align*}
\botsmash{\left<J_{r+1}({\eufb v}),{\bold d}_{\pi }
\boldsymbol\beta\right>_{_{_{_{\dsize\char"7E}}}}}(y_{r+1})
&=(T\vphantom{_x}\botsmash{\underset
{\dsize\char"7E}{\vphantom{_x}\botsmash{\beta}}})\bigl(J_{r+1}({\eufb v})(y_{r+1})\bigr) \\
&=(d/dt)\vphantom{_x}\botsmash{\underset
{\dsize\char"7E}{\vphantom{_x}\botsmash{\beta}}}\bigl(J_{r+1}(e^{t{\eufb v}})(y_{r+1})\bigr)(0).
\end{align*}
We recall that if some ${\eurb
u}\in T_{z}Z$ with $z=\pi_{r+1}(y_{r+1})$ is tangent to the curve
$\sigma^{{\eurb u}}_{z}(s)$ and if $y_{r+1}=j_{r+1}\upsilon (z)$,
then ${\vphantom{_x}\botsmash{({\bold d}_{t}f)_{_{_{_{\dsize\char"7E}}}}}}(y_{r+1})\,{\boldkey.}\,{\eurb u}
=(d/ds)(f\circ j_{r}\upsilon\circ\sigma^{{\eurb u}}_{z})(0)$.  Now
take $\left<J_{r}({\eufb v}),{\bold d}_{\pi }f\right>$  in place of $f$  to
obtain
\begin{align*}
({\bold d}_{t}\left<J_{r}({\eufb v}),{\bold d}_{\pi }f\right>)\,{\boldkey.}\,(y_{r+1},{\eurb u})
&=\botsmash{({\bold d}_{t}\left<J_{r}({\eufb v}),
{\bold d}_{\pi}f\right>)_{_{_{_{\dsize\char"7E}}}}}(y_{r+1})\;{\boldkey.}\;{\eurb u} \\
&=(d/ds)\,(\left<J_{r}({\eufb v}),{\bold d}_{\pi }f\right>\circ
j_{r}\upsilon \circ \sigma^{{\eurb u}}_{z})\,(0) \\
&=\left.(d/ds)\,(d/dt)\,f\circ
J_{r}(e^{t{\eufb v}}) \circ j_{r}\upsilon \circ \sigma^{{\eurb
u}}_{z}(s)\right|_{s=t=0}.
\end{align*}

On the other hand,
\begin{align*}
\left<J_{r+1}({\eufb v}),{\bold d}_{\pi }{\bold d}_{t}f\right>
{\boldkey.}\,y_{r+1},{\eurb u})
&=\botsmash{\left<J_{r+1}({\eufb v}),
{\bold d}_{\pi }\boldsymbol\beta
\right>_{_{_{_{\dsize\char"7E}}}}}(y_{r+1})\,{\boldkey.}\,{\eurb u} \\
&= (d/dt)\;\bigl(\;\botsmash{({\bold d}_{t}f)_{_{_{_{\dsize\char"7E}}}}}
\bigl(J_{r+1}(e^{t{\eufb v}})(y_{r+1})\bigr)\;{\boldkey.}\;{\eurb u}\;\bigr)\,(0).
\end{align*}
At this
stage it is necessary to put in the property of the prolongation
procedure, namely, $J_{r+1}(e^{t{\eufb v}})\circ j_{r+1}\upsilon
 =j_{r+1}(e^{t{\eufb v}}\circ \upsilon )$, in order to arrive at
\begin{align*}
\left<J_{r+1}({\eufb v}),{\bold d}_{\pi }{\bold d}_{t}f\right>{\boldkey.}\,
(y_{r+1},{\eurb u})
&= \left.(d/dt)\,(d/ds)\,f\circ j_{r}(e^{t{\eufb v}}
\circ \upsilon )\circ\sigma^{{\eurb u}}_{z}(s)\right|_{t=s=0} \\
&= \left.(d/dt)\,(d/ds)\,f\circ J_{r}(e^{t{\eufb v}})\circ
j_{r}\upsilon \circ \sigma^{{\eurb u}}_{z}(s)\right|_{t=s=0}.
\end{align*}
Thus both sides of (\ref{A6}) when evaluated at arbitrary
$(y_{r+1},{\eurb u})\in {\pi_{r+1}}^{-1}TZ$ provide one and the
same expression.
\end{proof}
\clearpage
\end{appendix}
{
\def\subtld#1{\vphantom{_{\scriptscriptstyle\#}}\smash{#1{\!}_{_{_{\textstyle\char"7E}}}}}
\begin{figure}
$$\xy
\xymatrix{
\vrule height14truemm width0pt
\\
&E
\ar[rrrrrr]^{{\mathsf g}}
\ar[ddl]+U+<-25truept,0truept>_{({\mathsf k}{\circ}{\mathsf g})_{_B}}
& & & & & &W
\ar[dddd]^{\chi}
\ar_{{\mathsf k}}[dddddddll]+UR-<5.8truept,0truept>
\ar[ddl]|{{\mathsf k}_{_X}}
\\
& &\save \POS+<-10truept,10truept>
   \drop+{{\fit g}^{\scriptscriptstyle-1}W}
   \ar[urrrrr]+<-14truept,-4.5truept>|(.4){\object+{\scriptstyle\chi^{-1}{\fit g}}}
   \ar[dll]+UR_<>(.6){{\fit g}^{-1}{\mathsf k}_{_X}}
   \POS="comment"
   \restore
   \ar[ul];"comment"^{{\mathsf g}_{_B}}
\\
{\fit g}'{}^{\scriptscriptstyle-1}W'\approx{\fit g}^{\scriptscriptstyle-1}{\fit k}^{\scriptscriptstyle-1}W'
\ar[rrrrrr]^<>(.5){({\fit k}^{\scriptscriptstyle-1}\chi')^{\scriptscriptstyle-1}{\fit g}}
\ar+D+<10truept,0truept>;[dddddrrrrr]+U-<4.4truept,0truept>|(.68){\object+{\scriptstyle\chi'{}^{-1}{\fit g}'}}
&&&&&&{\fit k}^{\scriptscriptstyle-1}W'
\ar[ddr]|<>(.257){\object+{}}|<>(.63){\object+{\scriptstyle{\fit k}^{-1}\chi'}}
\ar_(.25){\chi'{}^{-1}{\fit k}}[dddddl]+U+<1truept,0truept>
\\
\vrule height8pt width0pt
\\
&&&&\vrule height0pt width10truept&&&X
\ar[ddddll]^{{\fit k}}
\\
& B
\ar@{{}.>}^{{{\mathsf k}}_{_{\# }}{}{\eufb s}}[uuul]+D+<-25truept,0truept>
\ar@{{}.>}^{\eufb e}[uuuuu]
\ar@{{}.>}_{\eufb s=\overg}"comment"
\ar@{{}.>}^(.3){\frak g}[uuuuurrrrrr]
\ar^{{\fit g}}|<>(.181){\object+{}}|<>(.7079){\object+{}}|<>(.7756){\object+{}}[urrrrrr]
\ar@{{}.>}[ddrrrr]|{\object+{\scriptstyle{\mathsf k}\circ\frak g}}
\ar[dddrrrr]_{{\fit g}'}
\\
\vrule height8pt width0pt
\\
&&&&&W'
\ar[d]_<>(.3){\chi'}
\\
&\save\POS+<-30truept,+20truept>\drop+{\txt{
    $\eufb s={\mathsf g}_{\scriptscriptstyle\#}\eufb e$\\
    $\frak g={\mathsf g}\,{\scriptstyle\circ}\,\eufb e$\\
    $\subtld{({\mathsf k}_{\scriptscriptstyle\#}\eufb s)}
    ={\mathsf k}\,{\scriptstyle\circ}\,\frak g
    =\subtld{\big(({\mathsf k}\,{\scriptstyle\circ}\,{\mathsf g})_{\scriptscriptstyle\#}\eufb e\big)}$\\
    ${\mathsf k}_{\scriptscriptstyle\#}{\mathsf g}_{\scriptscriptstyle\#}\eufb e
    ={\mathsf k}_{\scriptscriptstyle\#}\eufb s=({\mathsf k}\,{\scriptstyle\circ}\,{\mathsf g})_{\scriptscriptstyle\#}\eufb e$
    }}\restore
&&&&X'
}\endxy$$
\caption{}
\label{Figure12}
\end{figure}
\clearpage
\begin{figure}
\xymatrixcolsep{0truepc}
\xymatrixrowsep{-.2truepc}
$$\xy
\xymatrix{
\vrule height20truemm width0pt
\\
\object+[o]{TZ}
\ar[rrrrrrr]|{\object+{\scriptstyle{\mathsf{id}}}}
\ar[ddrr]|{\object+{\scriptstyle\delta^T}}
\ar[ddddrrr]|<>(.4){\object+{\scriptstyle T\delta}}
\ar[dddddrr]_{\tau}
&
\vrule height0truept width40truept
&&&&&&\object+[o]{TZ}
\ar[dddddl]^{\tau}
\\
\vrule height10truept width0truept
\\
&&\delta^{\scriptscriptstyle-1}TB
\ar[rr]^<>(.5){\delta^{-1}\pi^T}
\ar[ddr]|{\object+{\scriptstyle\tau_{_B}{}^{-1}\delta}}
\ar[ddd]|<>(.557){\object+{}}
&&\delta^{\scriptscriptstyle-1}\pi^{\scriptscriptstyle-1}TZ
\ar[uurrr]|{\object+{\scriptstyle\tau^{-1}{\fit {id}}}}
\ar[ddr]|<>(.23){\object+{\scriptstyle(\pi^{-1}\tau){}^{-1}\delta}}|<>(.422){\object+{}}
\ar[dddll]|<>(.708){\object+{}}|<>(.794){\object+{}}
\\
\vrule height90truept width0truept
\\
&&&TB
\ar[uuuurrrr]|<>(.66){\object+{\scriptstyle T\pi}}
\ar[rr]^{\pi^T}
\ar[dddr]|{\object+{\scriptstyle\tau_{_B}}}
&&\pi^{\scriptscriptstyle-1}TZ
\ar[uuuurr]|<>(.4){\object+{\scriptstyle\tau^{-1}\pi}}
\ar[dddl]
\\
&&\object+[o]{Z}
\ar[rrrr]|<>(.226){\object+{}}|{\object+{{\fit {id}}}}|<>(.760){\object+{}}
\ar[ddrr]_{\delta}
&&&&\object+[o]{Z}
\\
\vrule height35truept width0truept
\\
&&&&\object+[o]{B}
\ar[uurr]_{\pi}
}\endxy$$
\caption{This is the complete picture underlying that of Fig.\mms\ref{Figure5}}
\label{Figure13}
\end{figure}
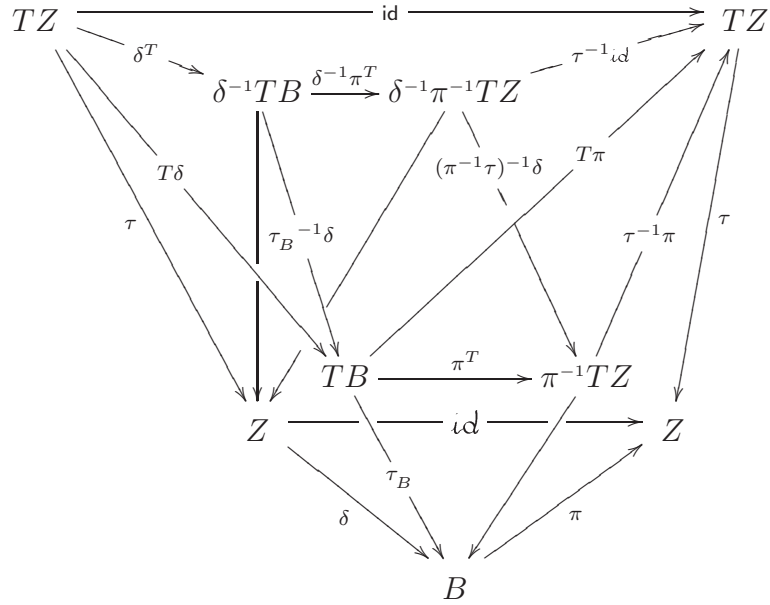
}
\clearpage
\def\t{{\overset\ast\tau}{}}
\def\b{\vphantom{_x}\botsmash{\underset{\displaystyle\char"7E}{\botsmash{\beta}}}}
\def\s{\scriptscriptstyle}
\def\lr{\overset\sim\leftrightarrow}
\xymatrixcolsep{15truept}
\begin{figure}
$$\xy
\xymatrix{
\save \POS+<+55truept,-20truept>
   \drop+{\txt{
        $\alpha\lr\beta$\\$\delta^{\s-1}\alpha\lr\delta^{\s-1}\beta$\\
        $\delta^\star=\delta^{\s\#}\delta^{\s-1}\alpha$\\
        $\wedge\delta^\ast{\scriptstyle\circ}\,(\delta^{\s-1}\wedge\pi^\ast)
        =\t^{\s-1}{\fit {id}}$
           }}
   \restore
&\vrule height0truept width90truept&
&\wedge T^\ast Z
\\
\vrule height40truept width0truept\\
&&Z
\ar@{{}.>}|{\object+{\scriptstyle\delta^\star\alpha}}[uur]+D+<-5truept,0truept>
\ar@{{}.{}}|{\delta^{-1}\beta}[r]
\ar@{{}.>}|{\delta^{-1}\alpha}[dr]+U+<-5truept,0truept>
\ar[ddddll]+U+<+9truept,0truept>^{\delta}
&{\;}
\ar@{{}.>}|<>(.248){\object+{}}[rr]
&&\delta^{\s-1}\pi^{\s-1}\wedge T^\ast Z
\ar[uull]+DR+<-7truept,0truept>_{\t^{-1}{\fit {id}}}
\ar[dll]+U+<+5truept,0truept>|{\object+{\scriptstyle\delta^{-1}\wedge\pi^\ast}}|<>(.701){\object+{}}
\ar[ddddl]^{(\pi^{-1}\t)^{-1}\delta}
\\
&&&\delta^{\s-1}\wedge T^\ast B
\ar[uuu]|{\object+{\scriptstyle\wedge\delta^\ast}}\\
\vrule height0truept width0truept\\
&&
\save \POS+<10truept,0truept>
   \drop+{\wedge T^\ast B}
   \POS="comment"
   \restore
   \ar[uur];"comment"|<>(.5){\object+{\scriptstyle\t_{_B}{\!}^{-1}\,\delta}}
\\
B
\ar+U+<+2truept,0truept>;[uuuuuurrr]+DL+<+7truept,0truept>^{\b}
\ar@{{}.>}|{\object+{\scriptstyle\alpha=\pi^{\s\#}\beta}}"comment"
\ar@{{}.>}|{\object+{\scriptstyle\beta}}[rrrr]
&&&&\pi^{\s-1}\wedge T^\ast Z
\ar"comment"|{\object+{\scriptstyle\wedge\pi^\ast}}
\ar[uuuuuul]+D+<+5truept,0truept>_(.7){\t^{-1}\pi}
}\endxy$$
\caption{}
\label{Figure14}
\vspace*{3cm}
\end{figure}
\def\sz{\scriptstyle}
\def\s{\scriptscriptstyle}
\def\lr{\overset\sim\leftrightarrow}
\xymatrixcolsep{-16.2969775truept}
\begin{figure}
$$
\xy
\xymatrix{
\delta^{\s-1}(VB)^\ast\otimes\wedge T^\ast Z\\
&\delta^{\s-1}(VB)^\ast\otimes\delta^{\s-1}\wedge T^\ast B
\ar[rr]
\ar|{\object+{\sz{\mathsf{id}}\,{\s\otimes}\,\wedge\delta^\ast}}[ul]
&&(VB)^\ast\otimes\wedge T^\ast B\\
&&\delta^{\s-1}(VB)^\ast\otimes\wedge T^\ast Z
\ar[rr]
\ar|{\object+{\sz{\mathsf{id}}\,{\s\otimes}\,\delta^{-1}\wedge\pi^\ast}}[ul]
&{\;}
\ar@{{}.{}}[u]
&(VB)^\ast\otimes\pi^{\s-1}\wedge T^\ast Z
\ar|{\object+{\sz{\mathsf{id}}\,{\s\otimes}\,\wedge\pi^\ast}}[ul]
\\
\save\POS+<0truept,13.5truept>\drop+{\txt{
    $\delta^\star\omega=\delta^{\s\#}\delta^{\s-1}\omega$\\
    $\omega\lr\bold d_\pi\beta$\\
    $\delta^{\s-1}\omega\lr\delta^{\s-1}\bold d\pi\beta$
}}\restore
&Z
\ar@{{}.>}^{\delta^\star\omega}[uuul]
\ar@{{}.>}|{\object+{\sz\delta^{-1}\omega}}[uu]
\ar@{{}.>}|{\object+{\sz\delta^{-1}\bold d_\pi\beta}}[ur]
\ar|{\object+{\sz\delta}}[rr]
&&B
\ar@{{}.{}}|{\object+{\sz\omega\,=\,\pi^{\#}\bold d_\pi\beta}}[u]
\ar@{{}.>}_{\bold d_\pi\beta}[ur]
}
\endxy$$
\caption{}
\label{Figure15}
\end{figure}
\clearpage
\fancyhead{} 
\fancyhead[CE,CO]{\slshape INTEGRATION BY PARTS IN
VARIATIONAL CALCULUS}
\nocite{Goldschmidt1973,Kazan,Pirani1979,Zharinov1992}

\enddocument